   \documentclass[twoside,reqno,11pt]{fcaa-var} %

\usepackage{graphicx}
\usepackage{epsfig}
\usepackage{amsthm}
\usepackage{amsmath}
\usepackage{latexsym}
\usepackage{amsfonts}
\usepackage{amssymb}

\newtheoremstyle{theorem}
  {15pt}          
  {15pt}  
  {\sl}  
  {\parindent}
  {\sc}  
  {. }   
  { }    
  {}     
\theoremstyle{theorem}
\newtheorem{lemma}{Lemma}[section]
\newtheorem{theorem}{Theorem}[section]
\newtheorem{corollary}{Corollary}[section]

\newtheoremstyle{defi}
  {15pt}          
  {15pt}  
  {\rm}  
  {\parindent}     
  {\sc}  
  {. }    
  { }    
  {}     
\theoremstyle{defi}
\newtheorem{definition}{Definition}[section]
\newtheorem{remark}{Remark}[section]

 \def\proofname{\indent\rm P\ r\ o\ o\ f.}
 
 \def\qed{\hfill$\Box$} 

 \usepackage{hyperref} 

\usepackage{comment}
\usepackage{bbm}
\usepackage{cleveref}
\newtheorem{property}{Property}[section]
\usepackage{xcolor}
\newcommand{\RN}[1]{%
  \textup{\uppercase\expandafter{\romannumeral#1}}%
  }

 \def\theequation{\arabic{section}.\arabic{equation}}

 \def\themycopyrightfootnote{\vspace*{3pt}
   \par  \noindent 
   \hfill  \vspace*{-36pt}
  }

 \title[ Fractional Diffusion Advection Reaction ]{On the Skewed Fractional Diffusion Advection Reaction Equation on the Interval}
 \author[\normalsize Y. Li]{\normalsize  Yulong Li}

\begin{document}

 \vbox to 2.5cm { \vfill }


 \bigskip \medskip

 \begin{abstract}
 	This article provides techniques of raising the regularity of fractional order equations and resolves fundamental questions on the one-dimensional homogeneous boundary-value problem of skewed (double-sided)  fractional diffusion advection reaction equation (FDARE) with variable coefficients on the bounded interval.   The existence of the true (classical) solution together with norm estimation is established and the precise regularity bound is found; also, the structure of the solution is unraveled, capturing the essence of regularity, singularity, and other features of the solution.  
 	
 	The key analysis lies in  exploring the properties of Gauss hypergeometric functions, solving  coupled Abel integral equations and dominant singular integral equations, and connecting the functions from  fractional Sobolev spaces to  the ones from H$\ddot{\text{o}}$lderian  spaces that admit integrable singularities at the endpoints. 
 	
 \medskip


 \smallskip

 \end{abstract}

 \maketitle




\section{Main results}
Under conditions
	\begin{equation}\label{conditionforthemainresult}
	\begin{cases}
	0<\alpha, \beta <1, \alpha+ \beta=1, 0<\mu<1,\\
	f(x)\in H^*(\Omega), p(x),q(x), k(x)\in C^1(\overline{\Omega}), \\
	p'(x), k'(x) \in H(\overline{\Omega}),\\
	k(x)>0 \, \text{on}\, \overline{\Omega}, \frac{q}{k}-\frac{1}{2}(\frac{p}{k})'\geq 0 \, \text{on}\, \overline{\Omega},\\
	\pi(1-\mu^2)\cot((1+\mu)\pi/2) +4(b-a)\|\frac{k'}{k}\|_{L^\infty(\Omega)}<0,
	\end{cases}
	\end{equation}

we consider the problem
	\begin{equation}\label{equationformainresult}
\begin{cases}
[L(u)](x)= f(x),\, x\in \Omega=(a,b),\\
u(a)=u(b)=0,\\
[L(u)](x) := -Dk(x)(\alpha\, {_aD_x^{-(1-\mu)}}+\beta\,{_xD_b^{-(1-\mu)}})Du\\
\qquad \qquad+p(x)Du+q(x)u(x).
\end{cases}
\end{equation}

We say a function $u(x)$ is a true (classical) solution to \eqref{equationformainresult} if $u\in AC(\overline{\Omega})$, $u(a)=u(b)=0$, $Du\in C(\Omega)$, $(\alpha\, {_aD_x}^{-(1-\mu)}+\beta\,{_xD_b}^{-(1-\mu)})Du \in C^1(\Omega)$ and $[L(u)](x)= f(x)$  $\forall x\in \Omega$.

\vspace{0.2cm}
Our main result is the following:
\begin{theorem}\label{theorem}
	Let conditions \eqref{conditionforthemainresult} be satisfied. Then in the Sobolev space $\widehat{H}^{(1+\mu)/2}_0(\Omega)$, there exists a unique true solution $u(x)$ to  \eqref{equationformainresult} (up to the equivalence of functions). And it is representable by 
	\[
	u(x)= {_aD_x^{-t}J_t},\quad x\in \Omega,
	\]
	$	J_t(x) \in H^*(\Omega)$ (depending on $t$)  provided that $t<1+\mu$.
\end{theorem}

We will also show that $1+\mu$ is optimal by giving a counter example, namely, the representation $	u(x)= {_aD_x^{-t}J_t}$ can fail for any $t>1+\mu$, therefore, Theorem \ref{theorem} is sharp up to the endpoint.

As a byproduct,  Corollary \ref{cor:1} clarifies that the homogeneous boundary-value problems \eqref{equationformainresult} and \eqref{equationformainresult-RL} that involve different types of fractional derivatives always have the same classical solution;   Corollary \ref{cor:Thesecondone} illustrates the behaviour of $\mu$-th order derivative and first derivative of the true solution (i.e., ${_aD_x^\mu}u, {_xD_b^\mu}u, Du$) near the boundary points, more precisely, under suitable conditions, ${_aD_x^\mu}u$ and ${_xD_b^\mu}u$ always vanish at $x=a$ and $x=b$ respectively;  meanwhile, $Du$ either vanishes or blows up at $x=a,b$ meaning that it does not admit  non-zero values at the endpoints, which suggests  difference from integer-order diffusion equations.
\section{Introduction}\label{sec:introduction}
The boundary-value problem of fractional diffusion advection reaction equation (FDARE) is one of the fundamental problems in the subject of fractional-order differential equations. In the last decade, the FDARE has been widely discussed in journals of  numerical analysis, applied analysis  and applied physics, and its applications have been found  in different scientific areas. Compared to  integer-order differential equations, fractional-order differential equations exhibit new features and open many opportunities in modelling various phenomena in physics. However, it turns out that the theoretical analysis of fractional-order equations can be very challenging and  that many fundamental questions still remain open.  Therefore, more additional attention is deserved towards the development of this subject.

We will focus on the homogeneous boundary-value problem of  FDARE in this article, as a part of our goal of systematic investigation of FDARE,  this work is a continuation of our previous work \cite{glsima18}, in which the skewed FDARE was investigated on the whole real axis. In the whole real line $\mathbb{R}$, the  properties of the solution to  FDARE behave  regularly and similarly to the ones of the solution to integer-order diffusion equations,  however, this is not the case on the  bounded interval.
Among extensive papers on the skewed FDARE, pioneering work and  excellent discussion  that are most related to our work appeared in \cite{ernmpde06} (2006),   \cite{MR3802435} (2018) and \cite{MR4024334} (2019). Therefore this paper is regarded as a  development of these fundamental work from a unified point of view.

During the study of FDARE and in order to obtain a satisfying discussion, the following  deserves  attention:

1.  The results of FDARE on the interval $(0,1)$ can not be directly generalized to that on an arbitrary interval  $(a,b)$ by  a simple transformation of the variable. Scaling of the length of  the domain is actually coupled with the variable coefficients. In this work, we are interested in studying FDARE on the general interval $(a,b)$.

2. The variational formulation of FDARE needs to be constructed with caution such that it does not  intertwine with raising the regularity at the current stage, which is a simple idea but an important start to the whole analysis.

3. The smoothness of  coefficients and the source function of FDARE usually can not guarantee the smoothness of  weak solution, what are the appropriate functional spaces to raise the regularity of weak solution and to seek the true solution needs to be considered.

\section{Strategy of the proof}\label{sec:SP}

The strategy for the whole proof of Theorem \ref{theorem} is two big steps. First, we intend to establish the weak solution for an appropriate variational formulation. Secondly, we try to raise the regularity of weak solution in a certain sense to recover the classical solution. More precisely,

 1. due to the presence of the advection term and variable coefficients,   we will adopt a suitably modified variational formulation such that we are able to establish the existence of a weak solution and without a need to raise the regularity at this stage.  This is where we get started, and  later, by picking up some regularity for the weak solution we  convert this modified variational formulation equation back to  original FDARE pointwisely, which automatically implies that this weak solution is  a classical solution (true solution).

2. the main body of the work, which is also the most challenging part, is to convert this weak solution to the true solution and the challenging is attributed to two aspects. First,  the solution $u(x)$ to FDARE usually dramatically lack regularity at the boundary regardless of the smoothness of coefficients and right-hand side function of FDARE, which suggests that the true solution $u(x)$ should be sought in functional spaces that admit singularities at the endpoints; Second, only limited regularity of the solution can be picked up in the context of fractional Sobolev spaces, which is usually not enough  for a  pointwise restoration of FDARE and results in difficulty in carrying the mathematical techniques performed in integer-order elliptic equations to the analysis of FDARE (for example, difference quotient technique is not any more applicable, etc.), new framework has to be developed.  

Taking into account above philosophy,  the analysis is led to:  we establish the existence of weak solution $u(x)$  in the usual fractional Sobolev space $\widehat{H}_0^{(1+\mu)/2}(\Omega)$, however, raise the regularity of ${_aD_x^\mu}u$  (therefore raising the regularity of $u(x)$) in the H$\ddot{\text{o}}$lderian  space that admits integrable singularities at the boundary, namely $H^*(\Omega)$, to better spaces $H^*_t(\Omega) (0<t<1)$; and in the middle, we  need to connect this two spaces $\widehat{H}_0^{(1+\mu)/2}(\Omega)$ and $H^*(\Omega)$ by showing that ${_aD_x^\mu}u$ which is originally from $\widehat{H}_0^{(1-\mu)/2}(\Omega)$ is actually located in $H^*(\Omega)$.
 Roughly speaking, we have essentially ``three" steps:
\[
u\in \widehat{H}_0^{(1+\mu)/2}(\Omega) \longrightarrow {_aD_x^\mu}u\in H^*(\Omega)\longrightarrow {_aD_x^\mu}u \in \bigcap_{0<t<1}H^*_t(\Omega).
\]
After integrating ${_aD_x^\mu}u$ by ${_aD_x^{-\mu}}$, we have 
\[
u(x)={_aD_x^{-(\mu+t)}}J_t, J_t \in H^*(\Omega), 0<t<1,
\]
which will  yield the representation in Theorem \ref{theorem} by adding  the case $t\leq 0$.
\section{Organization of the work}
Readers can perform quick search either through section titles or Definition, Property,  Lemma, Theorem and Corollary numbers.

\begin{itemize}
	\item  Section \ref{Notations} is about convention and notation.
	\item   Section \ref{sec:FRLO} gathers and lists necessary definitions and associated properties, most of which have been given specific citation information for readers' convenience. They will be extensively invoked during subsequent proofs.
	\item  Section \ref{sec:variationalformulation} is to  set up a suitable variational formulation and establish the existence of weak solution and norm estimation.
	\item  In Section \ref{sec:GeneralizedAbelequations}, three lemmas on coupled Abel integral equations will be established, which will serve the counter example at the end of Section \ref{sec:proofoftheorem} and Corollary \ref{cor:Thesecondone}  in Section \ref{sec:applications}.
	\item In Section \ref{sec:raisingtheregularity}, another three lemmas on raising the regularity will be established, which are  key steps towards the whole proof of Theorem \ref{theorem}.
	\item Section \ref{sec:proofoftheorem} provides the whole proof of Theorem \ref{theorem} and an example.
	\item Section \ref{sec:applications} gives two corollaries and proposes a question.
\end{itemize}

\section{Convention and notation}\label{Notations}

Convention
\begin{itemize}
\item    $\Omega=(a,b), \overline{\Omega}=[a,b]$ and $-\infty <a,b<\infty$, whenever they appear throughout the material.
\item We shall often not distinguish $``="$ at every point from $``="$ almost every point  in those equations when there is no chance of misunderstanding. On some occasions, the  notation $\overset{a.e.}{=}$  will be used  for  emphasizing the validity for ``almost everywhere"  to draw the reader's attention for those cases that are of importance.
\item  Whenever we deal with a function $f(x)$ belonging to Sobolev or $L^p$ spaces, it is implicitly assumed that $f$ denotes a suitable representative of the equivalence classes, unless otherwise specified.
\item All the functions are default to be real-valued, and all the constants that will appear in different contexts will be assumed to be real constants.
\end{itemize}

Notation
\begin{itemize}
\item $H^\lambda(\overline{\Omega})$: H$\ddot{\text{o}}$lderian  space ($\lambda>0$).
\vspace{0.2cm}
\item $H(\overline{\Omega}):=\underset{\lambda>0}{\bigcup}H^\lambda(\overline{\Omega})$.
\vspace{0.2cm}
\item $H^\lambda_0(\overline{\Omega}):=\{f:f(x)\in H^\lambda(\overline{\Omega}), f(a)=f(b)=0\}$.
\vspace{0.2cm}
\item $H^\lambda(\rho, \overline{\Omega})$: weighted H$\ddot{\text{o}}$lderian  space, ($\rho(x)=\prod_{k=1}^{n}|x-x_k|^{\mu_k}, x_k, x\in \overline{\Omega} $, $n$ is a positive integer).
\vspace{0.2cm}
\item  $H^\lambda_0(\rho, \overline{\Omega}):=\{f: f(x)\in H^\lambda(\rho, \overline{\Omega}), \rho(x_k) f(x_k)=0, k=1,\cdots, n\}$.
\vspace{0.2 cm}
\item  
$
H_0^\lambda(\epsilon_1,\epsilon_2):=\{f:f(x)=\frac{g(x)}{(x-a)^{1-\epsilon_1}(b-x)^{1-\epsilon_2}}, g(x)\in H^\lambda_0(\overline{\Omega})\}.
$
\vspace{0.1cm}
\item $
H^*(\Omega):=\underset{\underset{0<\lambda\leq 1}{0<\epsilon_1,\epsilon_2}}{\cup} H_0^\lambda(\epsilon_1,\epsilon_2).
$
\vspace{0.2cm}
\item  $
H^*_\sigma(\Omega):=\underset{\underset{\sigma<\lambda\leq 1}{0<\epsilon_1,\epsilon_2}}{\cup} H_0^\lambda(\epsilon_1,\epsilon_2).
$
\vspace{0.2cm}
\item ${_a}D_{x}^{-\sigma}$, ${_x}D_{b}^{-\sigma}$, $\boldsymbol{D}^{-\sigma}$ and $\boldsymbol{D}^{-\sigma * }$ represent fractional integrals if $\sigma>0$, identity operators if $\sigma=0$ and fractional derivatives if $\sigma<0$ (see specific definitions in Section \ref{sec:FRLO}).
\vspace{0.1cm}
\item $AC(\overline{\Omega})$: the set of absolutely continuous functions on $\overline{\Omega}$.
\item  $C(G)$: the set of all continuous functions on a  set $G$.
\item $C^n(G):=\{f: f^{(n)}(x)\in C(G)\}$.
\item Both $Df$ and $f'$ represent the usual derivative of a function $f$.
\item $(f , g)_\Omega$ and $(f ,g )_\mathbb{R}$ denote the  integrals $\int_\Omega fg$ and $\int_\mathbb{R} fg$, respectively.
\item $C_0^\infty(\Omega)$ consists of all the infinitely  differentiable functions on $\Omega$ and with compact support in $\Omega$.
\end{itemize}
\section{Prerequisite Knowledge}\label{sec:FRLO}
Necessary preliminaries that will be of use are presented first, including definitions and associated properties. These definitions and properties are known and most of them are quoted from the literature directly. 
Properties given in this section may not be in the strongest forms, but they are adequate for our purpose.

%
\subsection{Riemann-Liouville integrals and their properties}
\label{ssec:FRLI}
\begin{definition} \label{def:RLI}
Let $w:(c,d) \rightarrow \mathbb{R}, (c,d) \subset \mathbb{R}$ and  $\sigma >0$. The left and right Riemann-Liouville fractional integrals of order $\sigma$ are, formally respectively, defined as
\begin{align}
({_a}D_{x}^{-\sigma}w)(x) &:= \dfrac{1}{\Gamma(\sigma)}\int_{a}^{x}(x-s)^{\sigma -1}w(s) \, {\rm d}s, \label{eq:LRLI}\\
({_x}D_{b}^{-\sigma} w)(x) &:= \dfrac{1}{\Gamma(\sigma)}\int_{x}^{b}(s-x)^{\sigma-1}w(s) \, {\rm d}s,  \label{eq:RRLI}
\end{align}
where $\Gamma(\sigma)$ is  Gamma function. For convenience, when $c=-\infty$ or $d=\infty$ we set
\begin{equation}\label{eq:LRRLI}
(\boldsymbol{D}^{-\sigma}w)(x) :={_{-\infty}}D_{x}^{-\sigma} w \text{ and }
(\boldsymbol{D}^{-\sigma * }w)(x) :={_{x}}D_{\infty}^{-\sigma} w. 
\end{equation}
\end{definition}
In particular, if $\sigma=0$, ${_a}D_{x}^{-\sigma}$, ${_x}D_{b}^{-\sigma}$, $\boldsymbol{D}^{-\sigma}$ and $\boldsymbol{D}^{-\sigma * }$ are regarded as identity operators. 
\begin{property}[\cite{MR1347689}, eq. (2.72), (2.73), p. 48]\label{bounded-property}
Given $\sigma> 0$, fractional operators ${_aD_x^{-\sigma}}$ and ${_xD^{-\sigma}_b}$ are bounded in $L^p(\Omega) (p\geq 1)$:
\begin{equation}
\|{_aD_x^{-\sigma}}\psi\|_{L^p(\Omega)}\leq K\|\psi\|_{L^p(\Omega)}, \, \|{_xD^{-\sigma}_b}\psi\|_{L^p(\Omega)}\leq K\|\psi\|_{L^p(\Omega)}, \, K=\dfrac{(b-a)^\sigma}{\sigma\Gamma(\sigma)}.
\end{equation}
\end{property}
\begin{property}[\cite{MR1347689}, eq. (2.19), p. 34]\label{p-reflection}
Let $\sigma>0$ and $(Qf)(x)=f(a+b-x)$, then the following operators are reflective:
\begin{equation}
Q{_aD_x^{-\sigma}}Q={_xD_b^{-\sigma}}.
\end{equation}
\end{property}
\begin{property}\label{pro-mappings}
	If $0<\sigma<1$, $1<p<1/\sigma$, then the fractional operators ${_aD_x^{-\sigma}}$, ${_xD_b^{-\sigma}}$ are bounded from $L^p(\Omega)$ into $L^q(\Omega)$ with $q=\frac{p}{1-\sigma p}$.
\end{property}
\begin{property}\label{pro:MappingIntoHolderianFromLp}
	If $0<\frac{1}{p}<\sigma<1+\frac{1}{p}$, fractional operators ${_aD_x^{-\sigma}}$ and ${_xD^{-\sigma}_b}$  map the space $L^p(\Omega)$ into the H$\ddot{\text{o}}$lderian  space $H^{\sigma-1/p}(\overline{\Omega})$.
\end{property}
\begin{remark}
Property~\ref{pro-mappings} is a combination of Theorem 3.5 (\cite{MR1347689}, p. 66) and Property~\ref{p-reflection},  and Property~\ref{pro:MappingIntoHolderianFromLp} is a combination of Corollary of Theorem 3.6 (\cite{MR1347689}, p. 69) and Property~\ref{p-reflection}.
\end{remark}
\subsection{Riemann-Liouville derivatives and their properties}
\begin{definition}\label{def:RLD}
Let $w:(c,d) \rightarrow \mathbb{R}, (c,d) \subset \mathbb{R}$ and  $\sigma >0$. Assume  $n$ is the smallest integer greater than $\sigma$ (i.e., $n-1 \leq \sigma < n$). The left and right Riemann-Liouville fractional derivatives of order $\sigma$ are, formally respectively, defined as
\[
({_a}D_x^{\sigma}w)(x) := \frac{{\rm d}^n}{{\rm d}x^n} {_a}D_x^{\sigma-n} w  \text{ and }~
({_x}D_b^{\sigma}w)(x) := (-1)^n \frac{{\rm d}^n}{{\rm d} x^n} {_x}D_b^{\sigma-n} w.
\]
For ease of notation, when $c=-\infty$ or $d=\infty$ we set
\begin{equation}
(\boldsymbol{D}^{\sigma} w )(x)= {_{-\infty}}D_{x}^{\mu}w \text{ and }
(\boldsymbol{D}^{\sigma*}w)(x) =  {_{x}}D^{\mu}_{\infty}w .\label{def:Infty}
\end{equation}
\end{definition}
 \begin{property}[\cite{glsima18}, Theorem 4.1]\label{thm:symmetry}
For $v, w \in C_0^\infty(\Omega)$ and $\sigma\geq 0$, it is true that
\begin{equation} 
\begin{aligned}
& (\boldsymbol{D}^{\sigma} v , \boldsymbol{D}^{\sigma} w )_{\mathbb{R}}  =
  (\boldsymbol{D}^{\sigma*} v , \boldsymbol{D}^{\sigma*} w )_{\mathbb{R}} = (2\pi)^{2\sigma}\int_\mathbb{R}|\xi|^{2\sigma} \widehat{v}(\xi) \overline{\widehat{w}(\xi)}  \,{\rm d}\xi,\\
 &(\boldsymbol{D}^{\sigma} v, \boldsymbol{D}^{\sigma*} w )_{\mathbb{R}} +(\boldsymbol{D}^{\sigma} w, \boldsymbol{D}^{\sigma*} v )_{\mathbb{R}}
 =2\cos(\sigma \pi)(\boldsymbol{D}^{\sigma} v , \boldsymbol{D}^{\sigma} w )_{\mathbb{R}}.
 \end{aligned}
 \end{equation}
 \end{property}
\begin{property}[\cite{glsima18}, Property 2.4] \label{prop:Boundedness}
Let $0<\sigma $ and  $w\in C_0^\infty(\mathbb{R})$, then $\boldsymbol{D}^\sigma w, \boldsymbol{D}^{\sigma*} w \in L^p(\mathbb{R})$ for any $1\leq p<\infty$.
\end{property}
\subsection{Hypergeometric function and properties}
\begin{definition}[\cite{MR1688958}, pp. 64, 65]\label{def:Hygeometric}
For $|x|<1$, the  hypergeometric function is defined by the series
\begin{equation}
_2F_1(\sigma_1,\sigma_2,\sigma_3;x)=\sum_{n=0}^\infty\dfrac{(\sigma_1)_n(\sigma_2)_n}{(\sigma_3)_nn!}x^n,
\end{equation}
and by analytic continuation elsewhere. One of such analytic continuation is given by
 
\begin{equation}
_2F_1(\sigma_1,\sigma_2,\sigma_3;x)=\dfrac{\Gamma(\sigma_3)}{\Gamma(\sigma_2)\Gamma(\sigma_3-\sigma_2)}\int_0^1 t^{\sigma_2-1}(1-t)^{\sigma_3-\sigma_2-1}(1-xt)^{-\sigma_1}\, dt
\end{equation}
if $0<\sigma_2<\sigma_3$.
\end{definition}

The Pochhammer symbol $(z)_n$  in above definition with integer $n$ means
\begin{equation}
(z)_n=z(z+1)\cdots(z+n-1),\, n= 1,2,\cdots, (z)_0=1.
\end{equation}

We put down only some of the properties that will be needed.

\begin{property}\label{interchange}
\begin{equation}
_2F_1(\sigma_1,\sigma_2,\sigma_3;x)={_2F_1}(\sigma_2,\sigma_1,\sigma_3;x).
\end{equation}
\end{property}
\begin{property}[\cite{MR1347689}, eq. (2.46), p. 41]\label{coupled-inte}
Let $\sigma_2>0$, $ \sigma_1, \sigma_3\in \mathbb{R}$, and $\psi(x)=(x-a)^{\sigma_2-1}(b-x)^{\sigma_3-1}$, then for $x\in \Omega$
\begin{equation}
{_aD^{-\sigma_1}_x}\psi= \dfrac{(b-a)^{\sigma_3-1}\Gamma(\sigma_2)}{\Gamma(\sigma_1+\sigma_2)}(x-a)^{\sigma_1+\sigma_2-1} {_2F_1}(1-\sigma_3,\sigma_2,\sigma_1+\sigma_2;\frac{x-a}{b-a}).
\end{equation}
\end{property}
\begin{property}[\cite{MR1688958}, Theorem 2.3.2, p. 78]\label{pro:decompostion-F}
\begin{equation}
\begin{aligned}
&{_2F_1}(\sigma_1,\sigma_2,\sigma_1+\sigma_2+1-\sigma_3;1-x)\\
&=A\cdot{_2F_1}(\sigma_1,\sigma_2,\sigma_3;x)+B\cdot x^{1-\sigma_3}{_2F_1}(1+\sigma_1-\sigma_3,1+\sigma_2-\sigma_3,2-\sigma_3;x),
\end{aligned}
\end{equation}
where
\begin{equation}
A=\frac{\Gamma(\sigma_1+\sigma_2+1-\sigma_3)\Gamma(1-\sigma_3)}{\Gamma(\sigma_1+1-\sigma_3)\Gamma(\sigma_2+1-\sigma_3)},\quad B=\frac{\Gamma(\sigma_3-1)\Gamma(\sigma_1+\sigma_2+1-\sigma_3)}{\Gamma(\sigma_1)\Gamma(\sigma_2)}.
\end{equation}
\end{property}
\begin{property}[\cite{MR1688958}, Theorem 2.2.5, p. 68]\label{pro:Euler}
\begin{equation}
{_2F_1}(\sigma_1,\sigma_2,\sigma_3;x)=(1-x)^{\sigma_3-\sigma_1-\sigma_2}{_2F_1}(\sigma_3-\sigma_1,\sigma_3-\sigma_2,\sigma_3;x).
\end{equation}
\end{property}
\subsection{Functional spaces $H^*(\Omega)$, $H^*_\sigma(\Omega)$ and  properties}\label{Sec-IFS}
 We list several mapping properties related to $H^*(\Omega)$ and $H^*_\sigma(\Omega)$, which play an important role in connecting the whole analysis of  this work.
\begin{property}[\cite{MR1347689}, Lemma 30.2, p. 618]\label{pro-WeightedMapping}
	Let $0<\sigma<1$, the weighted singular operator
	\begin{equation}
	(S_{\nu_a,\nu_b} f)(x)=\frac{1}{\pi}\int_a^b\left(\frac{x-a}{t-a}\right)^{\nu_a}\left(\frac{b-x}{b-t}\right)^{\nu_b}\frac{f(t)}{t-x}dt
	\end{equation}
	maps the space $H^*_\sigma(\Omega)$ into itself provided that $\sigma-1<\nu_a\leq  \sigma$ and $\sigma-1<\nu_b\leq \sigma$.
\end{property}
\begin{property}[\cite{MR1347689}, Theorem 13.14, p. 248]\label{pro-correspondence-1}
	Let $0<\sigma<1$, the fractional integration operators ${_aD_x^{-\sigma}}$ and ${_xD_b^{-\sigma}}$ map the space $H^*(\Omega)$ one-to-one and onto the space $H^*_\sigma(\Omega)$, respectively. Consequently, ${_aD_x^{-\sigma}}(H^*(\Omega))={_xD_b^{-\sigma}}(H^*(\Omega))$.
\end{property}
\begin{property}\label{Existence-Abel}
 Let $0<\sigma<1$, $\gamma_1, \gamma_2>0$, $\gamma_1{_aD_x^{-\sigma}}+\gamma_2{_xD_b^{-\sigma}}$ maps $H^*(\Omega)$ one-to-one and onto the space $H^*_\sigma(\Omega)$.
\end{property}
\begin{remark}
 Property~\ref{Existence-Abel} is a combination of  Theorem 30.7 (\cite{MR1347689}, p. 626) and Property \ref{pro-correspondence-1}.
\end{remark}
\subsection{Dominant singular integral and properties}
\begin{definition} The singular integral operator $S$ is formally defined as
\[
(S\psi)(x)=\frac{1}{\pi}\int_a^b \frac{\psi (t)}{t-x} dt, x\in \Omega,
\]
\end{definition}
the convergence being understood in the principal value sense.
\begin{property}[\cite{MR1347689}, Corollary 2, p. 208]\label{pro:CR}
Denote $r_a(x)=x-a$, $ r_b(x)=b-x, x\in \overline{\Omega}$.
\[
{_xD_b^{-\lambda}}(r_b^{-\lambda}S(r_b^\lambda\psi))=r_a^\lambda S(r_a^{-\lambda}{_xD_b^{-\lambda}}\psi)
\]
 is valid for $0<\lambda<1, \psi \in L^p(\Omega), p> 1$.
\end{property}
\begin{property} [\cite{MR1347689}, Theorem 11.1, p. 200]  \label{pro:BInW}
Let $n$ be a positive integer, $0<\lambda<1$, then the operator $S$ is bounded in the space $H^\lambda_0(\overline{\Omega})$, and in the weighted space $H^\lambda_0(\rho, \overline{\Omega})$,  $\rho(x)=\prod_{k=1}^{n}|x-x_k|^{\mu_k}, x_k, x\in \overline{\Omega} $, provided that $\lambda<\mu_k<\lambda+1  \,(k=1,\cdots, n)$.
\end{property}
\begin{property}\label{pro:SolvabilityofDSI}
Let $c_1, c_2$ be constants and $c_1^2+c_2^2\neq 0$. Denote 
$\frac{c_1-ic_2}{c_1+ic_2}=e^ {i\theta}$ and choose the value of $\theta$ so that $0\leq \theta<2\pi$.  Further denote spaces $X_1=H^*(\Omega)$, $X_2=H^*(\Omega)\cap C( (a,b] )$, $X_3=H^*(\Omega)\cap C([a,b))$  and $X_4=H^*(\Omega)\cap C([a,b])$, and
define $n_a$ and $n_b$ as follows:
\[
\begin{aligned}
n_a(X_1)&= n_a(X_2)=1, n_a(X_3)=n_a(X_4)=0;\\ 
n_b(X_1)&=n_b(X_3)=1, n_b(X_2)=n_b(X_4)=0.
\end{aligned}
\]
Consider  the problem  
\begin{equation}\label{equ:TheoremDSI}
\begin{cases}
c_1\psi(x) + \frac{c_2}{\pi}\int_a^b\frac{\psi(t)}{t-x} dt=f(x), x\in \Omega,\\
 \text{where}\, f(x)=\frac{f_*(x)}{(x-a)^{1-\nu_a}(b-x)^{1-\nu_b}}, \, f_*(x)\in H(\overline{\Omega}), \nu_a,\nu_b\in \mathbb{R}.
\end{cases}
\end{equation}
 Then each of  the following holds:
 \begin{enumerate}
 \item  If \[
 1-n_a(X_i)-\frac{\theta}{2\pi}<\nu_a,\, \frac{\theta}{2\pi}-n_b(X_i)<\nu_b,
 \]
 then
\eqref{equ:TheoremDSI} is  unconditionally solvable in  $X_i \,(i=1, 2, \text{or}\, 3)$ 
and its general solution $\psi_i(x)$ in $X_i \,(i=1, 2, \text{or}\, 3)$ is given by
\begin{equation}\label{eq-sloution-CSE}
\begin{aligned}
\psi_i(x)&=C(x-a)^{1-n_a(X_i)-\frac{\theta}{2\pi}}(b-x)^{\frac{\theta}{2\pi}-n_b(X_i)}+\frac{c_1f(x)}{c_1^2+c_2^2}-\\
&\frac{c_2}{\pi(c_1^2+c_2^2)}\int_a^b\left( \frac{x-a}{t-a}\right)^{1-n_a(X_i)-\frac{\theta}{2\pi}}\left(\frac{b-x}{b-t}\right)^{\frac{\theta}{2\pi}-n_b(X_i)} \frac{f(t)}{t-x} dt ,
\end{aligned}
\end{equation}
where 
\[
 C=0 \,\, \text{for}\, i=2, 3, \, \text{and C is an arbitrary constant for}\,i=1.
\]
\item If \eqref{equ:TheoremDSI} is solvable in  $X_4$, then the solution $\psi_4(x)$ is unique and is given by
\begin{equation}
\begin{aligned}
\psi_4(x)&=\frac{c_1f(x)}{c_1^2+c_2^2}-\\
&\frac{c_2}{\pi(c_1^2+c_2^2)}\int_a^b\left( \frac{x-a}{t-a}\right)^{1-n_a(X_4)-\frac{\theta}{2\pi}}\left(\frac{b-x}{b-t}\right)^{\frac{\theta}{2\pi}-n_b(X_4)} \frac{f(t)}{t-x} dt .
\end{aligned}
\end{equation}
 \end{enumerate}
\end{property}
\begin{remark}
	According to the statement of Property \ref{pro:SolvabilityofDSI}, it is worth noting that if a solution of \eqref{equ:TheoremDSI} belongs to the space $X_4$, then it also lies in $X_i$ $(i=1,2,3)$ and therefore it has four equivalent representations which differ only in form. 
This property is  a special case of Theorem 30.2 on page 609 in \cite{MR1347689} by letting $Z_0(x)=1$, $a_1(x)=c_1$ and $a_2(x)=c_2$ in our case. And in part $(2)$ we  omitted the sufficient and necessary condition for  \eqref{equ:TheoremDSI} to be solvable in  $X_4$ since we shall not need it.  
\end{remark}

\subsection{Fractional Sobolev spaces and properties}
It is known that  there are various ways to define fractional Sobolev spaces, which are essentially equivalent but serve as convenient tools for deriving various properties under different contexts. On the whole real axis, one way is in terms of the Fourier transform as follows.

\begin{definition}\label{thm:FTHsR}
Given $0\leq s$, let
\begin{equation}
\widehat{H}^s(\mathbb{R}) := \left \{ w \in L^2(\mathbb{R}) : \int_{\mathbb{R}} (1 + |2\pi\xi|^{2s}) |\widehat{w}(\xi) |^2 \, {\rm d} \xi < \infty \right \}.
\end{equation}

It is endowed with semi-norm and norm
\[
|w|_{\widehat{H}^s(\mathbb{R})}:=\|(2\pi\xi)^s \widehat{w}\|_{L^2(\mathbb{R})},
\|w\|_{\widehat{H}^s (\mathbb{R})}:=\left(\|w\|^2_{L^2(\mathbb{R})} +|w|^2_{\widehat{H}^s(\mathbb{R})}\right)^{1/2}.
\]
\end{definition}
Another equivalent definition is achieved with the aid of left or right fractional-order weak derivative, which is a generalization of integer-order weak derivative:
\begin{definition}(\cite{glsima18}, Section 3)\label{Equivalent-definition}
	Given $0\leq s$ and assume $u(x)\in L^2(\mathbb{R})$, then  $u(x) \in \widehat{H}^s(\mathbb{R})$ if and only if there exists a unique $\psi_1(x) \in L^2(\mathbb{R})$  such that 
	\begin{equation}
	\int_\mathbb{R} u \cdot\boldsymbol{D}^s \psi =\int_\mathbb{R} \psi_1 \cdot \psi\quad 
	\end{equation}
	for any $\psi\in C^\infty_0(\mathbb{R})$.\\
	 Similarly, \\
	$u(x) \in \widehat{H}^s(\mathbb{R})$ if and only if there exists a unique $\psi_2(x) \in L^2(\mathbb{R})$  such that 
	\begin{equation}
 \int_\mathbb{R} u \cdot\boldsymbol{D}^{s*} \psi =\int_\mathbb{R} \psi_2 \cdot \psi
	\end{equation}
	for any $\psi\in C^\infty_0(\mathbb{R})$.
\end{definition}
\noindent
With above definitions, the following property can be deduced, which guarantees the existence of fractional derivatives and provides equivalent semi-norm and norm.
\begin{property}(\cite{glsima18}, Section 3)\label{weak-equivalence-norm}
	Assume $u\in \widehat{H}^s(\mathbb{R})$, $s\geq0$, then $\boldsymbol{D}^s u, \boldsymbol{D}^{s*}u$ exist a.e. and
	\begin{equation}
	|u|_{\widehat{H}^s(\mathbb{R})} = \|\boldsymbol{D}^s u\|_{L^2(\mathbb{R})}=\|\boldsymbol{D}^{s*}u\|_{L^2(\mathbb{R})}.
	\end{equation}
\end{property}
 Since in this work, we will mainly care about the fractional equation in finite domain, by restricting to the  bounded interval we can define the following analogue.
\begin{definition}\label{def:ForInterval}
	Given $0\leq s$.
	\begin{equation}
	\widehat{H}^s_0(\Omega):=\{\text{Closure of}\, \, u\in C_0^\infty(\Omega) \, \text{with respect to norm}\, \|\tilde{u}\|_{\widehat{H}^s(\mathbb{R})} \},
	\end{equation}
	where notation $\tilde{u}$ denotes the extension of $u(x)$ by $0$ outside $\Omega$. It is endowed with semi-norm and norm
	\[
	|u|_{\widehat{H}^s_0(\Omega)}:=|\tilde{u}|_{\widehat{H}^s(\mathbb{R})},
	\|u\|_{\widehat{H}^s_0(\Omega)}:=\|\tilde{u}\|_{\widehat{H}^s(\mathbb{R})}.
	\]
\end{definition}
It is well-known that $\widehat{H}^s(\mathbb{R})$ is a Hilbert space and so is $\widehat{H}^s_0(\Omega)$.
%

%
%
We shall also utilize another two useful facts:
\begin{property}\label{pro:alternate-form}
	Given $\frac{1}{2}<s<1$, then $u\in \widehat{H}^s_0(\Omega)$ can be represented as
	\begin{equation}
	u(x)={_aD_x^{-s}\psi_1}={_xD_b^{-s}}\psi_2,
	\end{equation}
	for certain   $\psi_1$, $\psi_2 \in L^2(\Omega)$. As a consequence, ${_aD_x^s}u$ and ${_xD_b^s}u$ exist a.e. and coincide with $\psi_1$, $\psi_2$, respectively.
\end{property}

\begin{property}\label{normproperty}
Given $1/2<s<1$, $g(x)\in C^1(\overline{\Omega})$, then there exists a positive constant $C$ such that
\begin{equation}
\tilde{g}\tilde{u}\in \widehat{H}^s(\mathbb{R})\quad \text{and}\quad
\|\tilde{g}\tilde{u}\|_{\widehat{H}^s(\mathbb{R})}\leq  C \|\tilde{u}\|_{\widehat{H}^s(\mathbb{R})}
\end{equation}
for any $u(x)\in \widehat{H}^s_0(\Omega)$. (Notation $\tilde{\cdot}$ denotes the extension by zero outside  $\Omega$.)
\end{property}
%
%
%
%
%
%
%
%
%
%
%
%
%
%
%
%
%
\section{Variational Formulation}\label{sec:variationalformulation}
Recall the conditions
	\begin{equation}\label{condition}
\begin{cases}
0<\alpha, \beta <1, \alpha+ \beta=1, 0<\mu<1,\\
f(x)\in H^*(\Omega), p(x),q(x), k(x)\in C^1(\overline{\Omega}), \\
p'(x), k'(x) \in H(\overline{\Omega}),\\
k(x)>0 \, \text{on}\, \overline{\Omega}, \frac{q}{k}-\frac{1}{2}(\frac{p}{k})'\geq 0 \, \text{on}\, \overline{\Omega},\\
\pi(1-\mu^2)\cot((1+\mu)\pi/2) +4(b-a)\|\frac{k'}{k}\|_{L^\infty(\Omega)}<0,
\end{cases}
\end{equation}
and  the problem
\begin{equation}\label{eq2}
\begin{cases}
[L(u)](x)= f(x),\, x\in \Omega=(a,b),\\
u(a)=u(b)=0,\\
[L(u)](x) := -Dk(x)(\alpha\, {_aD_x^{-(1-\mu)}}+\beta\,{_xD_b^{-(1-\mu)}})Du\\
\qquad \qquad+p(x)Du+q(x)u(x).
\end{cases}
\end{equation}

This section is to set up a proper variational formulation for  problem~\eqref{eq2}, so that we can establish the existence of a weak solution and do not require raising any regularity at this stage. To do so,  we begin with considering the operator $\tilde{L}$:
\begin{equation}
[\tilde{L}(u)](x) := -Dk(x)D(\alpha\, {_aD_x^{-(1-\mu)}}+\beta\,{_xD_b^{-(1-\mu)}})u+p(x)Du+q(x)u(x).
\end{equation}
(note the difference between operators $L$ and $\tilde{L}$)
and construct the suitable bilinear form in Definition~\ref{def-bilinear}, which is obtained from the left-hand side of \begin{equation}
\int_\Omega \frac{[\tilde{L}(u)](t)}{k(t)}\, v(t)\, dt=\int_\Omega \frac{f(t)}{k(t)} \, v(t)\, dt, \, v(t)\in C_0^\infty(\Omega). 
\end{equation}

 Only later,  we take care to raise the regularity to convert  the weak solution to the true solution and operator $\tilde{L}$ will also automatically become $L$.

\vspace{0.2cm}
Denote $s=(1+\mu)/2, s'=(1-\mu)/2$ throughout this section (Section \ref{sec:variationalformulation}).
\begin{definition}\label{def-bilinear}
Define  the bilinear form $B_2[\cdot ,\cdot ]$ on  the space $\widehat{H}^s_0(\Omega)$  as 
\begin{equation}
\begin{aligned}
B_2[u,v]&:=-\alpha ({_aD_x^s}u,{_xD_b^s}v)_\Omega-\beta({_xD_b^s}u, {_aD_x^s}v)_\Omega\\
&-\alpha(\frac{k'}{k}{_aD_x^\mu} u, v)_\Omega+\beta(\frac{k'}{k}{_xD_b^\mu} u, v)_\Omega\\
&+({_aD_x^{s'}}u, {_xD_b^s}(\frac{p}{k}v))_\Omega  +(\frac{q}{k}u,v)_{\Omega},
\end{aligned}
\end{equation}
for $u,v\in \widehat{H}^s_0(\Omega)$.
\end{definition}

Each term is well-defined under conditions~\eqref{condition} by paying an attention to that  ${_xD_b^s}(\frac{p}{k}v)$ exists a.e. and  belongs to  $L^2(\Omega)$ by Property \ref{normproperty} and Property \ref{weak-equivalence-norm}.
\begin{lemma}\label{Lem:Boundedness}
Under conditions~\eqref{condition}, there exist positive constants $Q_1,Q_2$ such that
\begin{equation}\label{ine:1}
|B_2[u,v]|\leq Q_1\|u\|_{\widehat{H}^s_0(\Omega)}\|v\|_{\widehat{H}^s_0(\Omega)}
\end{equation}
and
\begin{equation}\label{ine:2}
B_2[u,u]\geq Q_2\|u\|_{\widehat{H}^s_0(\Omega)}^2
\end{equation}
for all $u,v\in \widehat{H}^s_0(\Omega)$.
\end{lemma}
\begin{proof}

1. Let us  first prove \eqref{ine:1} and \eqref{ine:2} for all $u,v\in C_0^\infty(\Omega)$.

Assume $u,v\in C_0^\infty(\Omega)$. 

From Definition~\ref{def-bilinear} we readily check 
\begin{equation}
\begin{aligned}
|B_2[u,v]|&\leq \alpha\|{_aD_x^s}u\|_{L^2(\Omega)}\|{_xD_b^s}v\|_{L^2(\Omega)}+\beta\|{_xD_b^s}u\|_{L^2(\Omega)}\|{_aD_x^s}v\|_{L^2(\Omega)}\\
&+\alpha\|\frac{k'}{k}\|_{L^\infty(\Omega)}\|{_aD_x^\mu} u\|_{L^2(\Omega)}\|v\|_{L^2(\Omega)}+\beta\|\frac{k'}{k}\|_{L^\infty(\Omega)}\|{_xD_b^\mu} u\|_{L^2(\Omega)}\|v\|_{L^2(\Omega)}\\
&+|({_aD_x^{s'}}u, {_xD_b^s}(\frac{p}{k}v))_\Omega| +\|\frac{q}{k}\|_{L^\infty(\Omega)}\|u\|_{L^2(\Omega)}\|v\|_{L^2(\Omega)}.
\end{aligned}
\end{equation}
Now we examine little pieces above separately.
\begin{equation}
\begin{aligned}
&\alpha\|{_aD_x^s}u\|_{L^2(\Omega)}\|{_xD_b^s}v\|_{L^2(\Omega)}+\beta\|{_xD_b^s}u\|_{L^2(\Omega)}\|{_aD_x^s}v\|_{L^2(\Omega)}\\
&\leq \alpha\|u\|_{\widehat{H}^s_0(\Omega)}\|v\|_{\widehat{H}^s_0(\Omega)}+\beta\|u\|_{\widehat{H}^s_0(\Omega)}\|v\|_{\widehat{H}^s_0(\Omega)}\\
&=\|u\|_{\widehat{H}^s_0(\Omega)}\|v\|_{\widehat{H}^s_0(\Omega)},
\end{aligned}
\end{equation}
\begin{equation}
\begin{aligned}
&\alpha\|\frac{k'}{k}\|_{L^\infty(\Omega)}\|{_aD_x^\mu} u\|_{L^2(\Omega)}\|v\|_{L^2(\Omega)}+\beta\|\frac{k'}{k}\|_{L^\infty(\Omega)}\|{_xD_b^\mu} u\|_{L^2(\Omega)}\|v\|_{L^2(\Omega)}\\
&\leq \alpha\|\frac{k'}{k}\|_{L^\infty(\Omega)}\|u\|_{\widehat{H}^s_0(\Omega)}\|v\|_{\widehat{H}^s_0(\Omega)}+\beta\|\frac{k'}{k}\|_{L^\infty(\Omega)}\|u\|_{\widehat{H}^s_0(\Omega)}\|v\|_{\widehat{H}^s_0(\Omega)}\\
&=\|\frac{k'}{k}\|_{L^\infty(\Omega)}\|u\|_{\widehat{H}^s_0(\Omega)}\|v\|_{\widehat{H}^s_0(\Omega)},
\end{aligned}
\end{equation}
and 
\begin{equation}
\begin{aligned}
&|({_aD_x^{s'}}u, {_xD_b^s}(\frac{p}{k}v))_\Omega| +\|\frac{q}{k}\|_{L^\infty(\Omega)}\|u\|_{L^2(\Omega)}\|v\|_{L^2(\Omega)}\\
&= |({_aD_x^{-\mu}}{_aD_x^s}u, {_xD_b^s}(\frac{p}{k}v))_\Omega|+\|\frac{q}{k}\|_{L^\infty(\Omega)}\|u\|_{L^2(\Omega)}\|v\|_{L^2(\Omega)}\\
&\leq \dfrac{(b-a)^\mu}{\mu\Gamma(\mu)} ||{_aD_x^s}u||_{L^2(\Omega)}\|{_xD_b^s}(\frac{p}{k}v)\|_{L^2(\Omega)}+\|\frac{q}{k}\|_{L^\infty(\Omega)}\|u\|_{L^2(\Omega)}\|v\|_{L^2(\Omega)}\\
 &(\text{using Property~\ref{bounded-property}})\\
&\leq C\|u\|_{\widehat{H}^s_0(\Omega)}\|v\|_{\widehat{H}^s_0(\Omega)}+\|\frac{q}{k}\|_{L^\infty(\Omega)}\|u\|_{\widehat{H}^s_0(\Omega)}\|v\|_{\widehat{H}^s_0(\Omega)} \\
 &(\text{using Property~\ref{normproperty}})\\
&=(C+\|\frac{q}{k}\|_{L^\infty(\Omega)})\|u\|_{\widehat{H}^s_0(\Omega)}\|v\|_{\widehat{H}^s_0(\Omega)}.
\end{aligned}
\end{equation}
Putting them together we obtain
\begin{equation}\label{equ:boundednessforB2}
|B_2[u,v]|\leq Q_1\|u\|_{\widehat{H}^s_0(\Omega)}\|v\|_{\widehat{H}^s_0(\Omega)},
\end{equation}
for some appropriate positive constant $Q_1$.\\
2. For \eqref{ine:2},  we consider $B_2[u,u]$. 

Simplifying the term $({_aD_x^{s'}}u, {_xD_b^s}(\frac{p}{k}u))_\Omega$,
\[
({_aD_x^{s'}}u, {_xD_b^s}(\frac{p}{k}u))_\Omega=(u', \frac{p}{k}u)_{\Omega}=-\frac{1}{2}((\frac{p}{k})'u,u)_\Omega,
\]
  and  we find that
\begin{equation}\label{Inequality-VF-0}
\begin{aligned}
B_2[u,u]&=-\alpha ({_aD_x^s}u,{_xD_b^s}u)_\Omega-\beta({_xD_b^s}u, {_aD_x^s}u)_\Omega\\
&-\alpha(\frac{k'}{k}{_aD_x^\mu} u, u)_\Omega+\beta(\frac{k'}{k}{_xD_b^\mu} u, u)_\Omega\\
&+((\frac{q}{k}-\frac{1}{2}(\frac{p}{k})') u,u)_\Omega.
\end{aligned}
\end{equation}

Again, we examine little pieces above separately.

First, using the second identity in Property~\ref{thm:symmetry}, we have
\begin{equation}
\begin{aligned}
&-\alpha ({_aD_x^s}u,{_xD_b^s}u)_\Omega-\beta({_xD_b^s}u, {_aD_x^s}u)_\Omega\\
&=-\alpha(\boldsymbol{D}^{s} \tilde{u}, \boldsymbol{D}^{s*} \tilde{u} )_{\mathbb{R}}-\beta(\boldsymbol{D}^{s*} \tilde{u} , \boldsymbol{D}^{s} \tilde{u} )_{\mathbb{R}}\\
& \text{($\tilde{u}(x) $ is the extension of $u(x)$ by $0$ outside $\Omega$)}\\
&=-\alpha\cos(s \pi)(\boldsymbol{D}^{s} u , \boldsymbol{D}^{s} u )_{\mathbb{R}}-\beta \cos(s \pi)(\boldsymbol{D}^{s} u , \boldsymbol{D}^{s} u )_{\mathbb{R}}\\
&=-\cos (s\pi)|u|^2_{\widehat{H}^s_0(\Omega)}.
\end{aligned}
\end{equation}
Second,
\begin{equation}\label{Inequality-VF-1}
\begin{aligned}
&|-\alpha(\frac{k'}{k}{_aD_x^\mu} u, u)_\Omega+\beta(\frac{k'}{k}{_xD_b^\mu} u, u)_\Omega| \\
&\leq \|\frac{k'}{k}\|_{L^\infty(\Omega)}\|\alpha{_aD_x^\mu} u-\beta{_xD_b^\mu} u\|_{L^2(\Omega)}\|u\|_{L^2(\Omega)}.
\end{aligned}
\end{equation}
By using Minkowsky inequality and Property~\ref{bounded-property}, we deduce
\begin{equation}
\begin{aligned}
&\|\alpha{_aD_x^\mu} u-\beta{_xD_b^\mu} u\|_{L^2(\Omega)}\\
&\leq \|\alpha{_aD_x^\mu} u\|_{L^2(\Omega)}+\|\beta{_xD_b^\mu} u\|_{L^2(\Omega)}\\
&=\|\alpha{_aD_x^{-s'}}({_aD^s_x}u)\|_{L^2(\Omega)}+\|\beta {_xD_b^{-s'}}({_xD^s_b}u)\|_{L^2(\Omega)}\\
&\leq  \dfrac{\alpha(b-a)^{s'}}{s'\Gamma(s')}\|{_aD^s_x}u\|_{L^2(\Omega)}+ \dfrac{\beta(b-a)^{s'}}{s'\Gamma(s')} \|{_xD^s_b}u\|_{L^2(\Omega)}\\
&\leq \dfrac{(b-a)^{s'}}{s'\Gamma(s')}|u|_{\widehat{H}^s_0(\Omega)},
\end{aligned}
\end{equation}
and 
\begin{equation}
\begin{aligned}
\|u\|_{L^2(\Omega)}&=\|{_aD_x^{-s}}({_aD_x^s u})\|_{L^2(\Omega)}\\
&\leq \dfrac{(b-a)^s}{s\cdot\Gamma(s)}\|{_aD_x^s u}\|_{L^2(\Omega)}\\
&\leq \dfrac{(b-a)^s}{s\cdot\Gamma(s)}|u|_{\widehat{H}^s_0(\Omega)}.
\end{aligned}
\end{equation}
Thus, inequality~\eqref{Inequality-VF-1} further becomes
\begin{equation}\label{inequality-combine}
|-\alpha(\frac{k'}{k}{_aD_x^\mu} u, u)_\Omega+\beta(\frac{k'}{k}{_xD_b^\mu} u, u)_\Omega|\leq \|\frac{k'}{k}\|_{L^\infty(\Omega)}\dfrac{b-a}{s's\Gamma(s')\Gamma(s)}|u|_{\widehat{H}^s_0(\Omega)}^2.
\end{equation}

Utilizing this inequality, \eqref{Inequality-VF-0} now becomes
\begin{equation}
\begin{aligned}
B_2[u,u]&\geq \left(-\cos (s\pi) -\|\frac{k'}{k}\|_{L^\infty(\Omega)}\dfrac{b-a}{s's\Gamma(s')\Gamma(s)}\right)|u|_{\widehat{H}^s_0(\Omega)}^2 \\
&+((\frac{q}{k}-\frac{1}{2}(\frac{p}{k})') u,u)_\Omega\\
&=\left(-\cos (s\pi) -\|\frac{k'}{k}\|_{L^\infty(\Omega)}\dfrac{(b-a)\sin(s\pi)}{\pi s's}\right)|u|_{\widehat{H}^s_0(\Omega)}^2 \\
&+((\frac{q}{k}-\frac{1}{2}(\frac{p}{k})') u,u)_\Omega\\
&\text{(applying the formula $\Gamma(z)\Gamma(1-z)=\frac{\pi}{\sin(z\pi)}$, $z$ is not an integer)}.
\end{aligned}
\end{equation}
In view of  the last two conditions in \eqref{condition} and the fact that the norm $\|\cdot\|_{\widehat{H}^s_0(\Omega)}$ and the semi-norm $|\cdot|_{\widehat{H}^s_0(\Omega)}$ are equivalent, we obtain
\begin{equation}\label{equ:coervicityforB2}
B_2[u,u]\geq Q_2 \|u\|_{\widehat{H}^s_0(\Omega)}^2,
\end{equation}
for some appropriate positive constant $Q_2$.

3. Last, let us consider the general case for $u, v\in \widehat{H}^s_0(\Omega)$.

we claim that for any $u, v\in \widehat{H}^s_0(\Omega)$ there exist Cauchy sequences $\{u_n\}$, $\{v_n\}\subset C_0^\infty(\Omega)$ in  $\widehat{H}^s_0(\Omega)$ such that 
\begin{equation}\label{dequ:limit}
B_2[u,v]=\lim_{n\rightarrow \infty} B_2[u_n,v_n].
\end{equation}

To see this, assume $u, v\in \widehat{H}^s_0(\Omega)$, 	since $C_0^\infty(\Omega)$ is dense in $\widehat{H}^s_0(\Omega)$,  there exist Cauchy sequences $\{u_n\}, \{v_n\}\subset C_0^\infty(\Omega)$  such that 
\[
\lim_{n\rightarrow \infty} u_n=u, \lim_{n\rightarrow \infty} v_n=v
\]
with respect to $\|\cdot\|_{\widehat{H}^s_0(\Omega)}$.	

From Definition~\ref{def-bilinear} it is readily verified that 
\[
\begin{aligned}
 &\lim_{n\rightarrow \infty}-\alpha ({_aD_x^s}u_n,{_xD_b^s}v_n)_\Omega=   -\alpha ({_aD_x^s}u,{_xD_b^s}v)_\Omega,\\
 &\lim_{n\rightarrow \infty} -\beta({_xD_b^s}u_n, {_aD_x^s}v_n)_\Omega=-\beta({_xD_b^s}u, {_aD_x^s}v)_\Omega,\\
&\lim_{n\rightarrow \infty}-\alpha(\frac{k'}{k}{_aD_x^\mu} u_n, v_n)_\Omega=-\alpha(\frac{k'}{k}{_aD_x^\mu} u, v)_\Omega,\\
&\lim_{n\rightarrow \infty}\beta(\frac{k'}{k}{_xD_b^\mu} u_n, v_n)_\Omega=\beta(\frac{k'}{k}{_xD_b^\mu} u, v)_\Omega,\\
& \lim_{n\rightarrow \infty}(\frac{q}{k}u_n,v_n)_{\Omega}=(\frac{q}{k}u,v)_{\Omega}.
\end{aligned}
\]

For the term  $({_aD_x^{s'}}u, {_xD_b^s}(\frac{p}{k}v))_\Omega $, 
\[
\begin{aligned}
&|({_aD_x^{s'}}u_n, {_xD_b^s}(\frac{p}{k}v_n))_\Omega-({_aD_x^{s'}}u, {_xD_b^s}(\frac{p}{k}v))_\Omega|\\
&=|({_aD_x^{s'}}(u_n-u), {_xD_b^s}(\frac{p}{k}v_n))_\Omega-({_aD_x^{s'}}u, {_xD_b^s}(\frac{p}{k}v_n-\frac{p}{k}v))_\Omega|\\
& \leq \| {_aD_x^{s'}}(u_n-u)\|_{L^2(\Omega)}  \,\|  {_xD_b^s}(\frac{p}{k}v_n)\|_{L^2(\Omega)}+\| {_aD_x^{s'}}u\|_{L^2(\Omega)}  \,\|  {_xD_b^s}(\frac{p}{k}v_n-\frac{p}{k}v)\|_{L^2(\Omega)}\\
&\leq C_1  \|u_n-u\|_{\widehat{H}_0^s(\Omega)} \|v_n\|_{\widehat{H}_0^s(\Omega)}+C_2 \|u\|_{\widehat{H}_0^s(\Omega)} \|v_n-v\|_{\widehat{H}_0^s(\Omega)} \\
&\quad \text{(using Property \ref{normproperty} and noting $s'<s$)},\\
\end{aligned}
\]
for some positive constants $C_1, C_2$  not depending on $n$, therefore 
\[
\lim_{n\rightarrow \infty}({_aD_x^{s'}}u_n, {_xD_b^s}(\frac{p}{k}v_n))_\Omega=({_aD_x^{s'}}u, {_xD_b^s}(\frac{p}{k}v))_\Omega ,
\] 
and we see \eqref{dequ:limit}.

As a result,  it follows from    \eqref{equ:boundednessforB2}, \eqref{equ:coervicityforB2} and \eqref{dequ:limit}  that \eqref{ine:1} and \eqref{ine:2} hold for all $u,v\in \widehat{H}^s_0(\Omega)$. This completes the whole proof.
\end{proof}
Once we have shown that  $B_2[u,v]$ is bounded and coercive on $\widehat{H}^s_0(\Omega)$, applying  the Lax-Milgram theorem gives the following existence of weak solution.
\begin{lemma}\label{lem:Existence-weak}
Under conditions~\eqref{condition}, there exists a unique element $u\in \widehat{H}^s_0(\Omega)$ such that
\begin{equation}\label{equ:variationalformulation}
 B_2[u,v]=(\frac{f}{k},v)_{\Omega}
\end{equation}
 for all $v\in \widehat{H}^s_0(\Omega)$, and there is a positive constant $C$ such that
\begin{equation}\label{equ:normestimation}
\|u\|_{\widehat{H}^s_0(\Omega)}\leq C\|{_aD_x^{-s}}\frac{f}{k}\|_{L^2(\Omega)}. 
\end{equation}
\end{lemma}
\begin{proof}
	1. Since $f\in H^*(\Omega)$ and $k\in C^1(\overline{\Omega})$, there exists a certain $p>1$ such that $\frac{f}{k} \in L^p(\Omega)$. Hence ${_aD_x^{-s}}\frac{f}{k}\in L^2(\Omega)$ is guaranteed by Property \ref{pro-mappings}.

2. Define the linear functional $F$: $\widehat{H}^s_0(\Omega)\rightarrow \mathbb{R}$ as:
	\[
	F(v)=({_aD_x^{-s}}\frac{f}{k},{_xD_b^s}v)_{\Omega}.
	\]
	
	On one hand,
	\[
	({_aD_x^{-s}}\frac{f}{k},{_xD_b^s}v)_{\Omega}=(\frac{f}{k},{_xD_b^{-s}}{_xD_b^s}v)_{\Omega}=(\frac{f}{k},v)_{\Omega}
	\]
	by fractional integration by parts (eq. (2.20), p. 34, \cite{MR1347689}) and Property \ref{pro:alternate-form}.
	
	On the other hand,
	\begin{equation}\label{equ:inequality}
	|F(v)|\leq \|	{_aD_x^{-s}}\frac{f}{k}\|_{L^2(\Omega)}\|{_xD_b^s}v\|_{L^2(\Omega)}\leq  \|	{_aD_x^{-s}}\frac{f}{k}\|_{L^2(\Omega)} \|v\|_{\widehat{H}^s_0(\Omega)}
	\end{equation}
	by H$\ddot{\text{o}}$lder inequality. Thereby,  $F(\cdot)$ is bounded on $\widehat{H}^s_0(\Omega)$.
	
3. 
	\eqref{equ:variationalformulation} follows immediately from the Lax-Milgram theorem and  Lemma \ref{Lem:Boundedness}, and  the estimation \eqref{equ:normestimation} follows from \eqref{equ:inequality} and \eqref{ine:2}.
\end{proof}
\section{Relationship between generalized Abel equations with constant coefficients}\label{sec:GeneralizedAbelequations}
In this section, we investigate the relationship between solutions of generalized Abel equations with constant coefficients. More precisely, we will establish in Lemma \ref{lem:solution-derivative} the relationship between $u$ and $v$, where $u$, $v$ solve equations of the form
\begin{equation}
\begin{aligned}
\gamma_1{_aD^{-t}_x}u+\gamma_2 \,{_xD^{-t}_b}u=f,\quad\quad
\gamma_1{_aD^{-t}_x}v+\gamma_2\,{_xD^{-t}_b}v=Df.
\end{aligned}
\end{equation}
Before that, we shall need Lemma \ref{lem-solution-1} and \ref{lem-solution-2}.
\begin{lemma}\label{lem-solution-1}
Let  $0<t<1$, $0<\gamma_1, \gamma_2$. There exists a unique solution $u(x)\in H^*(\Omega) $ to the equation
\[
\gamma_1{_aD^{-t}_x}u+\gamma_2\,{_xD^{-t}_b}u=1, x\in \Omega,
\]
 and $u(x)=c\cdot(x-a)^p(b-x)^q$, where 
\[c=\Gamma(-q)\left(\gamma_1 (b-a)^{1+p+q}\Gamma(p+1)\int_a^b (x-a)^{-q-1}(b-x)^{-p-1}\, dx \right)^{-1}\] 
and $p,q$ are uniquely determined by conditions
\begin{equation}\label{conditionpq}
 p+q=-t \quad \text{and}\quad
 \gamma_1 \sin(q\pi)=\gamma_2\sin(p\pi).
\end{equation}
\end{lemma}
\begin{proof}
1. We first  verify that $u(x)=c\cdot(x-a)^p(b-x)^q$ is a solution and will prove the uniqueness in the last step. Before proceeding, we simply see that our assumptions are valid in the lemma, namely,  \eqref{conditionpq} uniquely determines $p$ and $q$ and $c$ is well-defined. The first is straightforward to be checked  and the later is confirmed by observing that 
\[
0<\int_a^b (x-a)^{-q-1}(b-x)^{-p-1}\, dx<\infty
\]
since $(x-a)^{-q-1}(b-x)^{-p-1}$ is strictly positive for $x\in \Omega$ and is integrable over $\Omega$.

2.
Let $t_1=1+t+p+q$ and $\tilde{u}(x)= c_1(x-a)^{-q-1}(b-x)^{-p-1}$, where 
\[
c_1= (\int_a^b (x-a)^{-q-1}(b-x)^{-p-1}\, dx )^{-1}.
\]
 Note from conditions~\eqref{conditionpq} that $-1<p,q<0$, $-1<-q-1$. This allows us to apply Property~\ref{coupled-inte} to $\gamma_1{_aD^{-t}_x}u$ and ${_aD^{-t_1}_x}{\tilde{u}}$ separately to obtain
\begin{equation}\label{l1}
\gamma_1{_aD^{-t}_x}u=\gamma_1 c\dfrac{(b-a)^q\Gamma(p+1)}{\Gamma(t+p+1)}(x-a)^{t+p} {_2F_1}(-q,p+1,t+p+1;\frac{x-a}{b-a}),
\end{equation}
and 
\begin{equation}\label{l2}
{_aD^{-t_1}_x}{\tilde{u}}= c_1\dfrac{(b-a)^{-p-1}\Gamma(-q)}{\Gamma(t+p+1)}(x-a)^{t+p} {_2F_1}(p+1,-q,t+p+1;\frac{x-a}{b-a}).
\end{equation}
Recall now Property~\ref{interchange}, we know 
\[
{_2F_1}(-q,p+1,t+p+1;\frac{x-a}{b-a})={_2F_1}(p+1,-q,t+p+1;\frac{x-a}{b-a}).
\]
Comparing \eqref{l1} and \eqref{l2}, we see that $\gamma_1 {_aD^{-t}_x}u={_aD^{-t_1}_x}{\tilde{u}}$.

3. Similarly, we  show $\gamma_2\,{_xD^{-t}_b}u={_xD^{-t_1}_b}{\tilde{u}}$. 

According to Property~\ref{p-reflection} we know ${_xD^{-t}_b}u=Q{_aD_x^{-t}}Qu$ and ${_xD^{-t_1}_b}{\tilde{u}}=Q{_aD_x^{-t_1}Q\tilde{u}}$, where the operator $(Qf)(x):=f(a+b-x)$.
By direct calculation with Property~\ref{coupled-inte}, we obtain
\begin{equation}\label{l3}
\gamma_2 Q{_aD^{-t}_x}Qu= c \,\gamma_2 \dfrac{(b-a)^p\Gamma(q+1)}{\Gamma(t+q+1)}(b-x)^{t+q} {_2F_1}(-p,q+1,t+q+1;\frac{b-x}{b-a}),
\end{equation}
and
\begin{equation}\label{l4}
Q{_aD^{-t_1}_x}Q{\tilde{u}}= c_1\dfrac{(b-a)^{-q-1}\Gamma(-p)}{\Gamma(t+q+1)}(b-x)^{t+q} {_2F_1}(q+1,-p,t+q+1;\frac{b-x}{b-a}).
\end{equation}
Notice again by Property~\ref{interchange} that
\[
{_2F_1}(-p,q+1,t+q+1;\frac{b-x}{b-a})={_2F_1}(q+1,-p,t+q+1;\frac{b-x}{b-a}).
\]
Now compare \eqref{l3} and \eqref{l4}, we readily check that $\gamma_2\, {_xD^{-t}_b}u={_xD^{-t_1}_b}{\tilde{u}}$ in view of the second piece of conditions~\eqref{conditionpq} and the fact that $\Gamma(z)\Gamma(1-z)=\pi/\sin(z\pi)$, ($z$ is not an integer).

4. Consequently, 
\[
\gamma_1{_aD^{-t}_x}u+\gamma_2\,{_xD^{-t}_b}u= {_aD^{-t_1}_x}{\tilde{u}}+{_xD^{-t_1}_b}{\tilde{u}}.
\]
 Since actually $t_1=1$  by utilizing the first piece of conditions~\eqref{conditionpq}, 
\[
{_aD^{-t_1}_x}{\tilde{u}}+{_xD^{-t_1}_b}{\tilde{u}}=\int_a^x\tilde{u}+\int_x^b\tilde{u}=\int_a^b\tilde{u}=1.
\] 
 Therefore,  $\gamma_1{_aD^{-t}_x}u+\gamma_2\,{_xD^{-t}_b}u=1$, which confirms that $u(x)$ is a solution.

 5. The uniqueness of the solution to $\gamma_1{_aD^{-t}_x}u+\gamma_2\,{_xD^{-t}_b}u=1$ in the space $H^*(\Omega)$ is a direct consequence of Property~\ref{Existence-Abel}, provided that $1\in H^*_{t}(\Omega)$. To see this, and in order to conveniently check the definition of $ H^*_{t}(\Omega)$ (see Section \ref{Notations}),  we simply rewrite $1$  in  the form of
 \begin{equation}\label{equ:ExpressionFor1}
  1=\frac{(x-a)^{t+\epsilon}(b-x)^{t+\epsilon}}{(x-a)^{1-(1-t-\epsilon)}(b-x)^{1-(1-t-\epsilon)}},
 \end{equation}

 where $\epsilon$ is chosen such that 
 \[
 0<\epsilon<1-t.
 \]
 In the numerator of right-hand side of \eqref{equ:ExpressionFor1},  it can be verified that
 \begin{equation}
(x-a)^{t+\epsilon}, (b-x)^{t+\epsilon}\in H^{t+\epsilon}(\overline{\Omega})
 \end{equation}  
 with the aid of  the well-known auxiliary inequality
 \begin{equation}
 \frac{|y_1^y-y_2^y|}{|y_1-y_2|^y}\leq 1,\, (0\leq y\leq 1, 0<y_1, 0<y_2, y_1\neq y_2).
 \end{equation}
 
 Therefore, their product 
\[
 (x-a)^{t+\epsilon}(b-x)^{t+\epsilon} \in H_0^{t+\epsilon}(\overline{\Omega}) \quad \text{since it vanishes at }\quad x=a,b.
\]
 Hence, by the definition of $ H^*_{t}(\Omega)$, we see  $1\in H^*_{t}(\Omega)$, which completes the whole proof.
\end{proof}
Analogously, we have the following.
\begin{lemma}\label{lem-solution-2}
Let $0<t<1$ and $0<\gamma_1,\gamma_2$. Then one of the solutions to the equation 
\[
D(\gamma_1{_aD^{-t}_x}u+\gamma_2\,{_xD^{-t}_b})u=1, x\in \Omega
\]
is $u(x)=c\cdot D(x-a)^{p+1}(b-x)^{q+1}$, where 
\begin{equation}
c=\frac{(-p-t)(-p-t+1)\Gamma(t)}{(1-t)(2-t)\gamma_2 \, \Gamma(t+p+1)\Gamma(q+2)},
\end{equation}
 and  $p,q$ are uniquely determined by
\begin{equation}\label{lem:condition}
 p+q=-t \quad \text{and}\quad
 \gamma_1 \sin(q\pi)=\gamma_2\sin(p\pi).
\end{equation}
\end{lemma}
\begin{proof}
1. We begin with calculating
 \[
\gamma_1{_aD^{-t}_x}((x-a)^{p+1}(b-x)^{q+1}) \quad \text{and}\quad  \gamma_2\,{_xD^{-t}_b}((x-a)^{p+1}(b-x)^{q+1}).
\]

2. In view of Property~\ref{coupled-inte} and the condition $p+q=-t$,
\begin{equation}\label{equ:left-side}
\begin{aligned}
&\gamma_1{_aD^{-t}_x}((x-a)^{p+1}(b-x)^{q+1})\\
&=A\cdot \left(\frac{x-a}{b-a}\right)^{t+p+1} {_2F_1}(-q-1,p+2,t+p+2;\frac{x-a}{b-a}),
\end{aligned}
\end{equation}
where
\begin{equation}
A=\dfrac{\gamma_1(b-a)^{2}\Gamma(p+2)}{\Gamma(t+p+2)}.
\end{equation}

3. We continue to compute with the condition $p+q=-t$ 
\begin{equation}\label{equ:second-part}
\begin{aligned}
&\gamma_2\,{_xD^{-t}_b}((x-a)^{p+1}(b-x)^{q+1})\\
&=\gamma_2 Q{_aD^{-t}_x}Q((x-a)^{p+1}(b-x)^{q+1})\,(\text{recall}\,(Qf)(x):=f(a+b-x))\\
&=B\cdot(\frac{b-x}{b-a})^{t+q+1} {_2F_1}(-p-1,q+2,t+q+2;\frac{b-x}{b-a}),
\end{aligned}
\end{equation}
where
\[
B=\dfrac{\gamma_2(b-a)^{2}\Gamma(q+2)}{\Gamma(t+q+2)}.
\]

4. Now we take a closer look at the term ${_2F_1}(-p-1,q+2,t+q+2;\frac{b-x}{b-a})$. Applying Property~\ref{pro:decompostion-F} and the condition $p+q=-t$   yields
\begin{equation}\label{equ:2F1CloserLookAt}
\begin{aligned}
&{_2F_1}(-p-1,q+2,t+q+2;\frac{b-x}{b-a})\\
&=C\cdot {_2F_1}(-p-1,q+2,-p-t;\frac{x-a}{b-a})\\
&+ D \cdot (\frac{x-a}{b-a})^{p+t+1}{_2F_1}(t,3,t+p+2;\frac{x-a}{b-a}),
\end{aligned}
\end{equation}
where 
\[
C=\frac{\Gamma(t+q+2)\Gamma(t+p+1)}{\Gamma(t)\Gamma(3)},\quad D=\frac{\Gamma(-t-p-1)\Gamma(t+q+2)}{\Gamma(-p-1)\Gamma(q+2)}.
\]

5. Inserting \eqref{equ:2F1CloserLookAt} into ~\eqref{equ:second-part} and taking $p+q=-t$ into account, we thus arrive at
\begin{equation}\label{equ:C-D}
\begin{aligned}
&\gamma_2\,{_xD^{-t}_b}((x-a)^{p+1}(b-x)^{q+1})\\
&=E\cdot (\frac{b-x}{b-a})^{t+q+1}{_2F_1}(-p-1,q+2,-p-t;\frac{x-a}{b-a})\\
&+F\cdot (\frac{x-a}{b-a})^{p+t+1}(\frac{b-x}{b-a})^{t+q+1} {_2F_1}(t,3,t+p+2;\frac{x-a}{b-a}),
\end{aligned}
\end{equation}
where 
\[
E=\frac{\gamma_2 (b-a)^{2}\Gamma(t+p+1)\Gamma(q+2)}{\Gamma(t)\Gamma(3)},\,
F=\frac{\gamma_2(b-a)^{2}\Gamma(-t-p-1)}{\Gamma(-p-1)}.
\]

Notice $t+q+1=1-p$ and  we apply Property~\ref{pro:Euler} to \eqref{equ:C-D} to  obtain
\begin{equation}\label{equ:further}
\begin{aligned}
&\gamma_2\,{_xD^{-t}_b}((x-a)^{p+1}(b-x)^{q+1})\\
&=E\cdot{_2F_1}(1-t,-2,-p-t;\frac{x-a}{b-a})\\
&+ F \cdot (\frac{x-a}{b-a})^{p+t+1} {_2F_1}(p+2,t+p-1,t+p+2;\frac{x-a}{b-a}).
\end{aligned}
\end{equation}

Since
 \[
{_2F_1}(p+2,t+p-1,t+p+2;\frac{x-a}{b-a})={_2F_1}(-q-1,p+2,t+p+2;\frac{x-a}{b-a})
\] 
(by Property~\ref{interchange}), \eqref{equ:further} further becomes
\begin{equation}
\begin{aligned}
&\gamma_2\,{_xD^{-t}_b}((x-a)^{p+1}(b-x)^{q+1})\\
&=E\cdot{_2F_1}(1-t,-2,-p-t;\frac{x-a}{b-a})\\
&+F \cdot (\frac{x-a}{b-a})^{p+t+1}{_2F_1}(-q-1,p+2,t+p+2;\frac{x-a}{b-a}).
\end{aligned}
\end{equation}
Compare the last term with~\eqref{equ:left-side}, we therefore see that
\begin{equation}\label{equ:right-side}
\begin{aligned}
&\gamma_2\, {_xD^{-t}_b}((x-a)^{p+1}(b-x)^{q+1})\\
&=E\cdot{_2F_1}(1-t,-2,-p-t;\frac{x-a}{b-a})\\
&+F \cdot A^{-1}\cdot\gamma_1{_aD^{-t}_x}((x-a)^{p+1}(b-x)^{q+1}).
\end{aligned}
\end{equation}

6. Consequently, summing up \eqref{equ:left-side} and \eqref{equ:right-side} and simplifying by using the definition of Hypergeometric function  (Definition \ref{def:Hygeometric}) give
\begin{equation}
\begin{aligned}
&\gamma_1{_aD^{-t}_x}((x-a)^{p+1}(b-x)^{q+1})+\gamma_2\,{_xD^{-t}_b}((x-a)^{p+1}(b-x)^{q+1})\\
&=E\cdot{_2F_1}(1-t,-2,-p-t;\frac{x-a}{b-a})\\
&+(1+F\cdot A^{-1})\gamma_1{_aD^{-t}_x}((x-a)^{p+1}(b-x)^{q+1})\\
&=E\cdot \sum_{n=0}^2\dfrac{(1-t)_n(-2)_n}{(-p-t)_n\,n!}(\frac{x-a}{b-a})^n\quad (\text{every term is zero after $n=2$})\\
&+(1+F\cdot A^{-1})\gamma_1{_aD^{-t}_x}((x-a)^{p+1}(b-x)^{q+1}).
\end{aligned}
\end{equation}

7. It suffices to show that $1+ F \cdot  A^{-1}=0$  by taking advantage of the formula $\Gamma(z)\Gamma(1-z)=\pi/\sin(z\pi)$, ($z$ is not an integer). Namely, we check
\begin{equation}
\begin{aligned}
F\cdot A^{-1}&=\frac{\gamma_2 \Gamma(-1-t-p)\Gamma(t+p+2)}{\gamma_1\Gamma(-p-1)\Gamma(p+2)}\\
&=\frac{\gamma_2}{\gamma_1}\cdot\frac{\sin(p\pi)}{\sin((t+p)\pi)}\\
&=-\frac{\gamma_2}{\gamma_1}\cdot\frac{\sin(p\pi)}{\sin(q\pi)},
\end{aligned}
\end{equation}
thus  $1+ F \cdot  A^{-1}=0$ in view of the second piece of conditions~\eqref{lem:condition}. It follows that
\begin{equation}\label{lem:quadratic}
\begin{aligned}
&\gamma_1{_aD^{-t}_x}((x-a)^{p+1}(b-x)^{q+1})+\gamma_2\,{_xD^{-t}_b}((x-a)^{p+1}(b-x)^{q+1})\\
&=E\cdot \sum_{n=0}^2\dfrac{(1-t)_n(-2)_n}{(-p-t)_n\, n!}(\frac{x-a}{b-a})^n.
\end{aligned}
\end{equation}

8. Last, on one hand, differentiating \eqref{lem:quadratic} twice at both sides gives
\[
\begin{aligned}
&DD(\gamma_1{_aD^{-t}_x}+\gamma_2\,{_xD^{-t}_b})((x-a)^{p+1}(b-x)^{q+1})\\
&=\frac{(1-t)(2-t)\gamma_2 \Gamma(t+p+1)\Gamma(q+2)}{(-p-t)(-p-t+1)\Gamma(t)}.
\end{aligned}
\]

On the other hand, interchanging the order of differentiation and fractional integrations is permitted (\cite{MR1347689}, Theorem 2.2, p. 39) and results in
\[
\begin{aligned}
&DD(\gamma_1{_aD^{-t}_x}+\gamma_2\,{_xD^{-t}_b})((x-a)^{p+1}(b-x)^{q+1})\\
&=D(\gamma_1{_aD^{-t}_x}+\gamma_2\,{_xD^{-t}_b})D((x-a)^{p+1}(b-x)^{q+1}).
\end{aligned}
\]

Thereby,
\begin{equation}\label{equ:Lastone}
\begin{aligned}
&D(\gamma_1{_aD^{-t}_x}+\gamma_2\,{_xD^{-t}_b})D((x-a)^{p+1}(b-x)^{q+1})\\
&=\frac{(1-t)(2-t)\gamma_2 \Gamma(t+p+1)\Gamma(q+2)}{(-p-t)(-p-t+1)\Gamma(t)}.
\end{aligned}
\end{equation}
Dividing  both sides of  \eqref{equ:Lastone} by the right-hand side concludes
\begin{equation}
D(\gamma_1{_aD^{-t}_x}u+\gamma_2\,{_xD^{-t}_b})u=1,
\end{equation}
as desired, which completes the whole proof.
\end{proof}

Based on Lemma \ref{lem-solution-1} and \ref{lem-solution-2}, we arrive at the following:

\begin{lemma}\label{lem:solution-derivative}
Let   $0<\sigma<1$ and $0<\gamma_1, \gamma_2$. Assume that $f\in H^*_\sigma(\Omega)$, $Df$ exists in $\Omega$ and $Df\in H^*_\sigma(\Omega)$. If $u(x),v(x)\in H^*(\Omega)$ and satisfy 
\begin{equation}
(\gamma_1{_aD_x^{-\sigma}}+\gamma_2\,{_xD_b^{-\sigma}})u=f\,\text{and} \, (\gamma_1{_aD_x^{-\sigma}}+\gamma_2\,{_xD_b^{-\sigma}})v=Df, x\in \Omega,
\end{equation}
then
\begin{equation}
D(\gamma_1{_aD_x^{-\sigma}}+\gamma_2\,{_xD_b^{-\sigma}})(u-Y)=0, x\in \Omega,
\end{equation}
where 
\begin{equation}
\begin{aligned}
Y&=\int_a^x v(t)\, dt-\frac{cS}{S_1}\int_a^x (t-a)^p(b-t)^q\, dt+\frac{c_1S}{S_1} D((x-a)^{p+1}(b-x)^{q+1}),\\
S&= \int_a^bv(t) \, dt,\\
c&=  \Gamma(-q)\left(\gamma_1 (b-a)^{1+p+q}\Gamma(p+1)\int_a^b (t-a)^{-q-1}(b-t)^{-p-1}\, dt \right)^{-1},\\
S_1&=  \frac{\Gamma(-q)\int_a^b (t-a)^p(b-t)^q \, dt}{ \gamma_1 (b-a)^{1+p+q}\Gamma(p+1)\int_a^b (t-a)^{-q-1}(b-t)^{-p-1}\, dt },\\
c_1&= \frac{(-p-\sigma)(-p-\sigma+1)\Gamma(\sigma)}{(1-\sigma)(2-\sigma)\gamma_2 \,\Gamma(\sigma+p+1)\Gamma(q+2)},
\end{aligned}
\end{equation}
and $p$, $q$ are uniquely determined by
\begin{equation}
p+q=-\sigma \quad \text{and}\quad
\gamma_1 \sin(q\pi)=\gamma_2\sin(p\pi).
\end{equation}
\end{lemma}
\begin{proof}
1. According to the assumption in the lemma we first have
\begin{equation}\label{equ:lem-start-1}
D(\gamma_1{_aD_x^{-\sigma}}+\gamma_2\,{_xD_b^{-\sigma}})u-
(\gamma_1{_aD_x^{-\sigma}}+\gamma_2\,{_xD_b^{-\sigma}})v=0, x\in \Omega.
\end{equation}

 We examine the second term
\begin{equation}
\begin{aligned}
&(\gamma_1{_aD_x^{-\sigma}}+\gamma_2\,{_xD_b^{-\sigma}})v\\
&=\gamma_1{_aD_x^{-\sigma}}D{_aD^{-1}_x}v-\gamma_2\,{_xD_b^{-\sigma}}D{_xD^{-1}_b}v\\
&=\gamma_1 D {_aD_x^{-\sigma}}{_aD^{-1}_x}v-\gamma_2 D{_xD_b^{-\sigma}}{_xD^{-1}_b}v \\
&(\text{interchanging the order of operators by Theorem 2.2, p. 39, \cite{MR1347689}})\\
&=\gamma_1 D {_aD_x^{-\sigma}}{_aD^{-1}_x}v-\gamma_2 D{_xD_b^{-\sigma}}(S-{_aD^{-1}_x}v)\quad (\text{let}\, S=\int_a^bv(t)\, dt)\\
&=D(\gamma_1{_aD_x^{-\sigma}}+\gamma_2\,{_xD_b^{-\sigma}}){_aD^{-1}_x}v+\frac{\gamma_2 S}{\Gamma(\sigma)}(b-x)^{\sigma-1}.
\end{aligned}
\end{equation}

 Substituting this into \eqref{equ:lem-start-1} gives
\begin{equation}\label{equ:lem-start-2}
D(\gamma_1{_aD_x^{-\sigma}}+\gamma_2\,{_xD_b^{-\sigma}})(u-{_aD^{-1}_x}v)=\frac{\gamma_2 S}{\Gamma(\sigma)}(b-x)^{\sigma-1}.
\end{equation}

2. We claim
\begin{equation}\label{equ:AdditionalConstant}
D(\gamma_1{_aD_x^{-\sigma}}+\gamma_2\,{_xD_b^{-\sigma}})(c \int_a^x (t-a)^p(b-t)^q\, dt)=1-\frac{\gamma_2 S_1}{\Gamma(\sigma)}(b-x)^{\sigma-1},
\end{equation}
where 
\[
\begin{aligned}
&c=\Gamma(-q)\left(\gamma_1 (b-a)^{1+p+q}\Gamma(p+1)\int_a^b (t-a)^{-q-1}(b-t)^{-p-1}\, dt \right)^{-1},\\
& S_1=c\int_a^b (t-a)^p(b-t)^q \, dt,
\end{aligned}
\]
and $p,q$ satisfy
\[
 p+q=-\sigma \quad \text{and}\quad
\gamma_1 \sin(q\pi)=\gamma_2\sin(p\pi).
\]

To see this,
\[
\begin{aligned}
&D(\gamma_1{_aD_x^{-\sigma}}+\gamma_2\,{_xD_b^{-\sigma}})(c \int_a^x (t-a)^p(b-t)^q\, dt)\\
&=\gamma_1 D{_aD_x^{-\sigma}}(c \int_a^x (t-a)^p(b-t)^q\, dt)\\
&+\gamma_2 D{_xD_b^{-\sigma}}(S_1-c\int_x^b (t-a)^p(b-t)^q\, dt)\\
&=\gamma_1 {_aD_x^{-\sigma}}D(c \int_a^x (t-a)^p(b-t)^q\, dt)\\
&+\gamma_2 \,{_xD_b^{-\sigma}}D(-c\int_x^b (t-a)^p(b-t)^q\, dt)+\gamma_2 D{_xD_b^{-\sigma}}S_1\\
&(\text{interchanging the order of operators})\\
&=(\gamma_1{_aD_x^{-\sigma}}+\gamma_2\,{_xD_b^{-\sigma}})(c (x-a)^p(b-x)^q)-\frac{\gamma_2 S_1}{\Gamma(\sigma)}(b-x)^{\sigma-1}\\
&=1-\frac{\gamma_2 S_1}{\Gamma(\sigma)}(b-x)^{\sigma-1}\, (\text{applying  Lemma~\ref{lem-solution-1}}).
\end{aligned}
\]

3. Before going further, let us simply see that $c$ is well-defined, $c\neq 0$ and  $S_1\neq 0$  by observing that
\[
0<\int_a^b (t-a)^{-q-1}(b-t)^{-p-1}\, dt,\, \int_a^b (t-a)^p(b-t)^q \, dt<\infty.
\]

 Adding \eqref{equ:lem-start-2} to \eqref{equ:AdditionalConstant} multiplied by $\frac{S}{S_1}$, we have
\begin{equation}\label{equ:subtract-1}
\begin{aligned}
&D(\gamma_1{_aD_x^{-\sigma}}+\gamma_2\,{_xD_b^{-\sigma}})(u-{_aD^{-1}_x}v+\frac{cS}{S_1}\int_a^x (t-a)^p(b-t)^q\, dt)\\
&=\frac{S}{S_1}.
\end{aligned}
\end{equation}

 Invoking Lemma~\ref{lem-solution-2}, we know
\begin{equation}\label{equ:subtract-2}
D(\gamma_1{_aD_x^{-\sigma}}+\gamma_2\,{_xD_b^{-\sigma}})\left(\frac{-c_1S}{S_1} D(x-a)^{p+1}(b-x)^{q+1}\right)=-\frac{S}{S_1},
\end{equation}
where
\[
c_1=\frac{(-p-\sigma)(-p-\sigma+1)\Gamma(\sigma)}{(1-\sigma)(2-\sigma)\gamma_2 \,\Gamma(\sigma+p+1)\Gamma(q+2)}.
\]

Summing up \eqref{equ:subtract-1} and \eqref{equ:subtract-2} produces
\begin{equation}
D(\gamma_1{_aD_x^{-\sigma}}+\gamma_2\,{_xD_b^{-\sigma}})(u-Y)=0,
\end{equation}
where 
\[
Y=\int_a^x v(t)\, dt-\frac{cS}{S_1}\int_a^x (t-a)^p(b-t)^q\, dt+\frac{c_1S}{S_1} D(x-a)^{p+1}(b-x)^{q+1},
\]
which is the desired result. We will keep the coefficients this way since it is convenient to be used later.
\end{proof}
\section{Raising the regularity}\label{sec:raisingtheregularity}
In this section, we will establish three lemmas, which are the key analysis of the whole work and  crucial steps towards the proof of Theorem \ref{theorem}. Also, as explained in Section \ref{sec:SP}, this three lemmas will help us connect the weak solution from the Sobolev space $\widehat{H}^{(1+\mu)/2}_0(\Omega)$ to $H^*(\Omega)$ and further raise the regularity from the space $H^*(\Omega)$ to better spaces $H^*_\sigma(\Omega)$.
\vspace{0.2cm}

Let $0<\mu, \alpha, \beta <1, \alpha+\beta=1$ throughout this section (Section \ref{sec:raisingtheregularity}).

\begin{lemma}\label{lemma-Pre-RaiseRegularity}
	Let $A=\alpha-\beta\cos(\mu\pi)$, $B=\beta\sin(\mu\pi)$ and $f(x)\in H(\overline{\Omega})$. Denote  $r_b(x)=b-x, x\in \overline{\Omega}$, $F(x)=(b-x)^\mu f(x)$ and 
	$\frac{A-iB}{A+iB}=e^ {i\theta}$ with the value of $\theta$ chosen so that $0\leq \theta<2\pi$. 
	Consider the problem
	\begin{equation}\label{equ-RProblem-1}
	A\psi(x)+ \frac{B}{\pi}\int_a^b\frac{\psi(t)}{t-x} dt =F(x), x\in \Omega.
	\end{equation}
	Then each of the following is valid:
	\begin{enumerate}
		\item \eqref{equ-RProblem-1} is solvable in spaces $X_2=H^*(\Omega)\cap C( (a,b] )$ and $X_3=H^*(\Omega)\cap C([a,b))$ respectively, and its according solution $\psi_i   (i=2,3)$ is unique and is represented as
		\begin{equation}
		\begin{aligned}
		\psi_i(x)&=\frac{AF(x)}{A^2+B^2}-\\
		&\frac{B}{\pi(A^2+B^2)}\int_a^b\left( \frac{x-a}{t-a}\right)^{1-n_a(X_i)-\frac{\theta}{2\pi}}\left(\frac{b-x}{b-t}\right)^{\frac{\theta}{2\pi}-n_b(X_i)} \frac{F(t)}{t-x} dt ,
		\end{aligned}
		\end{equation}
		where
		\[
		\begin{aligned}
	 n_a(X_2)&=1, n_a(X_3)=0;\\ 
    n_b(X_2)&=0, n_b(X_3)=1.
		\end{aligned}
		\]
		\item $\theta$ satisfies 
		\begin{equation}\label{equ:thetacondition}
			\mu<\frac{\theta}{2\pi}<1.
		\end{equation}
		\item   The solution $\psi_i(x)$ in part $(1)$ satisfies
		\begin{equation}
		\frac{\psi_i(x)}{(b-x)^\mu}\in H^*(\Omega), \, ( i=2,3).
		\end{equation}
		\item  The solution $\psi_2(x)$ in part $(1)$ satisfies
		 \begin{equation}
		{_aD_x^{-\mu}} \frac{\psi_2}{r_b^\mu} \in L^{p}(I_b),
		\end{equation}	
		where $p=\frac{1}{1-\frac{\theta}{2\pi}},\,I_b=(\frac{b-a}{2}, b) $.
		
		\item  If the solution $\psi_i(x)$ in part $(1)$ satisfies $\psi_2(x)=\psi_3(x)$, $x\in \Omega$, then $\psi_2$ (or $\psi_3$, which is the same) has four equivalent representations:
		\begin{equation}
		\begin{aligned}
			\psi_2(x)&=\frac{AF(x)}{A^2+B^2}-\frac{B}{\pi(A^2+B^2)}\int_a^b\left( \frac{x-a}{t-a}\right)^{1-\frac{\theta}{2\pi}}\left(\frac{b-x}{b-t}\right)^{\frac{\theta}{2\pi}} \frac{F(t)}{t-x} dt ,\\
			&=\frac{AF(x)}{A^2+B^2}-\frac{B}{\pi(A^2+B^2)}\int_a^b\left( \frac{x-a}{t-a}\right)^{1-\frac{\theta}{2\pi}}\left(\frac{b-x}{b-t}\right)^{\frac{\theta}{2\pi}-1} \frac{F(t)}{t-x} dt ,\\
			&=\frac{AF(x)}{A^2+B^2}-\frac{B}{\pi(A^2+B^2)}\int_a^b\left( \frac{x-a}{t-a}\right)^{-\frac{\theta}{2\pi}}\left(\frac{b-x}{b-t}\right)^{\frac{\theta}{2\pi}} \frac{F(t)}{t-x} dt ,\\
			&=\frac{AF(x)}{A^2+B^2}-\frac{B}{\pi(A^2+B^2)}\int_a^b\left( \frac{x-a}{t-a}\right)^{-\frac{\theta}{2\pi}}\left(\frac{b-x}{b-t}\right)^{\frac{\theta}{2\pi}-1} \frac{F(t)}{t-x} dt .
		\end{aligned}
		\end{equation}
		
	\end{enumerate}
\end{lemma}
\begin{proof}

1. Proof for part (1). 

Let us see that   part (1) is just a direct consequence of  the first part of  Property~\ref{pro:SolvabilityofDSI}, we only need to justify the applicability of  Property~\ref{pro:SolvabilityofDSI}.

We shall need to check three aspects.

Firstly,  $A^2+B^2\neq 0$ by recalling $0<\mu, \alpha, \beta <1, \alpha+\beta=1$ (see the beginning of  Section \ref{sec:raisingtheregularity}).

Secondly, the function $F(x)$ can be equivalently rewritten as
\begin{equation}
F(x)=(b-x)^\mu f(x)=\frac{(b-x)^{\mu} f(x)}{(x-a)^{1-1}(b-x)^{1-1}}.
\end{equation}

 In the numerator,  since $f\in H(\overline{\Omega})$, there exists a $0<\lambda_0<\mu$ so that
 \begin{equation}\label{equ:lambda}
 f\in H^{\lambda_0}(\overline{\Omega}) ,
 \end{equation}
 and we directly check that 
 \[
 (b-x)^\mu \in H^{\mu}(\overline{\Omega}) 
 \]
 with  the assistance of the well-known inequality
 \begin{equation}\label{equ:wellknowninequality}
 \frac{|y_1^y-y_2^y|}{|y_1-y_2|^y}\leq 1,\, (0\leq y\leq 1, 0<y_1, 0<y_2, y_1\neq y_2).
 \end{equation}
 Therefore, their product $(b-x)^\mu f(x)\in H^{\lambda_0}(\overline{\Omega})$ follows.
 
Lastly,   observe that $B\neq 0$,  which implies that  $\theta\neq 0$. Hence, 
 \[
 1-n_a(X_i)-\frac{\theta}{2\pi}<1, \frac{\theta}{2\pi}-n_b(X_i)<1
 \]
hold for  $i=2,  3$. 

So, all the hypotheses are met for applying  the first part of  Property~\ref{pro:SolvabilityofDSI},   the part $(1)$ of Lemma \ref{lemma-Pre-RaiseRegularity} follows.

\vspace{0.2cm}

2. Now we  prove the part (2).

Since
\begin{equation}\label{equ-AB-Theta}
\frac{A-iB}{A+iB}=e^{i\theta}\, \text{and}\, 0\leq \theta<2\pi,
\end{equation}
it is clear that
\[\frac{\theta}{2\pi}<1.
\]

Substituting for $A, B$ and simplifying \eqref{equ-AB-Theta}, we have
\[
\frac{\alpha-\beta\cos(\mu\pi)-i\beta\sin(\mu\pi)}{\alpha-\beta\cos(\mu\pi)+i\beta\sin(\mu\pi)}=\frac{\alpha-\beta e^{i\mu \pi}}{\alpha-\beta e^{-i\mu \pi}}=\frac{\frac{\alpha}{\beta}-e^{i\mu\pi}}{\frac{\alpha}{\beta}-e^{-i\mu\pi}}=e^{i\theta}.
\]
Solving the last equality  for $\frac{\alpha}{\beta}$ gives
\[
\frac{\alpha}{\beta}=\frac{e^{i(\theta-2\mu\pi)}-1}{e^{i\theta}-1}e^{i\mu\pi}.
\]
Taking the fact $e^{iz}-1=2ie^{iz/2}\sin(z/2), z\in \mathbb{C}$ into account, we arrive at
\[
\frac{\alpha}{\beta}=\frac{\sin((\theta-2\mu\pi)/2)}{\sin(\theta/2)}.
\]
Again, recalling $0<\mu, \alpha, \beta <1, \alpha+\beta=1$   we derive

\[
0<(\theta-2\mu\pi)/2<\pi.
\]
Hence, we see the assertion
\[
\mu<\frac{\theta}{2\pi}<1.
\]

This fact will be  used in the rest of  proof and  in the proof of  subsequent lemmas.

\vspace{0.2cm}

3. To prove the part (3), namely $\frac{\psi_i(x)}{(b-x)^\mu}\in H^*(\Omega) \, ( i=2,3)$, we discuss the two cases  separately in this step and the next step.

For $i=2$, substituting for $\psi_2$ into $\frac{\psi_2(x)}{(b-x)^\mu}$ by using the representation in   part (1) of the lemma  and simplifying, we obtain
\begin{equation}\label{equ-ForPsiTwoExpression}
\begin{aligned}
 	\frac{\psi_2(x)}{(b-x)^\mu}&=\frac{Af(x)}{A^2+B^2}-\frac{B}{\pi(A^2+B^2)}\int_a^b\left( \frac{x-a}{t-a}\right)^{-\frac{\theta}{2\pi}}\left(\frac{b-x}{b-t}\right)^{\frac{\theta}{2\pi}-\mu} \frac{f(t)}{t-x} dt \\
 	&=\frac{A}{A^2+B^2}\Pi_1-\frac{B}{\pi(A^2+B^2)}\Pi_2.
 \end{aligned}
\end{equation}

Let us look at $\Pi_1$ and $\Pi_2$.

$\Pi_1$, namely $f(x)$ (recall $f\in H^{\lambda_0}(\overline{\Omega})$ in \eqref{equ:lambda}), is relatively easily seen to belong to $H^*(\Omega)$ by  manipulating as follows:
\begin{equation}\label{equ:fractionalpi1}
\Pi_1=\frac{(x-a)^\epsilon f(x)(b-x)^\epsilon}{(x-a)^{1-(1-\epsilon)}(b-x)^{1-(1-\epsilon)}}, 
\end{equation}
where  $\epsilon$ is chosen so that $ 0<\epsilon<\lambda_0$. 

In the numerator,  it can be directly verified with \eqref{equ:wellknowninequality} that
\begin{equation}
(x-a)^\epsilon,\,(b-x)^\epsilon\in H^\epsilon(\overline{\Omega}).
\end{equation}

 Therefore, the product 
\begin{equation}\label{equ-BelongtoHolderianSpace}
(x-a)^\epsilon f(x)(b-x)^\epsilon\in H_0^\epsilon(\overline{\Omega})
\end{equation}
by taking into account the boundary and $f\in H^{\lambda_0}(\overline{\Omega})$. 

In the denominator of \eqref{equ:fractionalpi1}, simply observe that $1-\epsilon>0$.

Hence, $\Pi_1\in H^*(\Omega)$ follows by the definition of $H^*(\Omega)$ (see notation in Section~\ref{Notations}).

Similarly, to see $\Pi_2\in H^*(\Omega)$, we rewrite
\begin{equation}\label{equ-UsedForLp}
\begin{aligned}
\Pi_2&=\int_a^b\left( \frac{x-a}{t-a}\right)^{-\frac{\theta}{2\pi}}\left(\frac{b-x}{b-t}\right)^{\frac{\theta}{2\pi}-\mu} \frac{f(t)}{t-x} dt \\
&=\frac{\int_a^b\left(\frac{x-a}{t-a}\right)^{\epsilon}\left(\frac{b-x}{b-t}\right)^{\theta/(2\pi)-\mu+\epsilon}\frac{(t-a)^{\theta/(2\pi)+\epsilon}f(t)(b-t)^\epsilon}{t-x} dt}{(x-a)^{1-(1-\theta/(2\pi)-\epsilon)}(b-x)^{1-(1-\epsilon)}},
\end{aligned}
\end{equation}
where $\epsilon$ is chosen so that $0<\epsilon<\text{min}\{\lambda_0, 1-\frac{\theta}{2\pi}\}$ (it should be clear this $\epsilon$ is different  from  the one in \eqref{equ:fractionalpi1} and we will use nation $\epsilon$ this way several times in the rest of proof).

In the denominator, $1-\theta/(2\pi)-\epsilon>0$, $1-\epsilon>0$.  

According to the definition of   $H^*(\Omega)$, $\Pi_2$ belongs to $ H^*(\Omega)$ is ensured provided that the numerator is in $H^{\epsilon/2}_0(\overline{\Omega})$, namely
\begin{equation}\label{equ-BelongtoHolderianSpace-1}
\int_a^b\left(\frac{x-a}{t-a}\right)^{\epsilon}\left(\frac{b-x}{b-t}\right)^{\theta/(2\pi)-\mu+\epsilon}\frac{(t-a)^{\theta/(2\pi)+\epsilon}f(t)(b-t)^\epsilon}{t-x} dt\in H^{\epsilon/2}_0(\overline{\Omega}).
\end{equation}

\eqref{equ-BelongtoHolderianSpace-1} is verified by a direct application of  Property~\ref{pro:BInW} by checking that 
\[
(t-a)^{\theta/(2\pi)+\epsilon}f(t)(b-t)^\epsilon\in H^{\epsilon/2}_0(\overline{\Omega}),
\]
which can be justified analogously to \eqref{equ-BelongtoHolderianSpace}, \\
and that  
\[
\epsilon/2<\epsilon<1+\epsilon/2, \quad \epsilon/2<\theta/(2\pi)-\mu+\epsilon<1+\epsilon/2,
\]
where the fact that $\frac{\theta}{2\pi}-\mu>0$ from  \eqref{equ:thetacondition} was used in the second piece.

Hence, 
\begin{equation}\label{equ:Pi2}
\Pi_2\in H^*(\Omega),
\end{equation} 
and therefore, 
\[
\frac{\psi_2(x)}{(b-x)^\mu} \in H^*(\Omega).
\]	 

\vspace{0.2cm}

4.  We continue to consider the case $i=3$.

Substituting for $\psi_3(x)$ into $	\frac{\psi_3(x)}{(b-x)^\mu}$  by using the representation  in   part $(1)$ of the lemma  and simplifying, we have
\begin{equation}
\begin{aligned}
\frac{\psi_3(x)}{(b-x)^\mu}&=\frac{Af(x)}{A^2+B^2}-\frac{B}{\pi(A^2+B^2)}\int_a^b\left( \frac{x-a}{t-a}\right)^{1-\frac{\theta}{2\pi}}\left(\frac{b-x}{b-t}\right)^{\frac{\theta}{2\pi}-1-\mu} \frac{f(t)}{t-x} dt \\
&=\frac{A}{A^2+B^2}\Sigma_1-\frac{B}{\pi(A^2+B^2)}\Sigma_2.
\end{aligned}
\end{equation}

For the first term, $\Sigma_1=f(x)$, which is the same as $\Pi_1$ in the previous step, thus, $\Sigma_1\in H^*(\Omega)$.

For $\Sigma_2$,
\begin{equation}
\begin{aligned}
\Sigma_2&=\int_a^b\left( \frac{x-a}{t-a}\right)^{1-\frac{\theta}{2\pi}}\left(\frac{b-x}{b-t}\right)^{\frac{\theta}{2\pi}-1-\mu} \frac{f(t)}{t-x} dt \\
&=\frac{\int_a^b\left(\frac{x-a}{t-a}\right)^{1-\theta/(2\pi)+\epsilon}\left(\frac{b-x}{b-t}\right)^{\epsilon}\frac{(t-a)^{\epsilon}f(t)(b-t)^{1-(\theta/(2\pi)-\mu-\epsilon)}}{t-x} dt}{(x-a)^{1-(1-\epsilon)}(b-x)^{1-(\theta/(2\pi)-\mu-\epsilon)}},
\end{aligned}
\end{equation}
where $\epsilon$ is chosen so that $0<\epsilon<\text{min} \{ \lambda_0, \frac{\theta}{2\pi}-\mu \}$.

Again, observe that  in the denominator $1-\epsilon>0$  and $\frac{\theta}{2\pi}-\mu-\epsilon>0$. 

Then by the definition of  $H^*(\Omega)$, $\Sigma_2$ belongs to $ H^*(\Omega)$ is guaranteed provided that the numerator is in $H^{\epsilon/2}_0(\overline{\Omega})$, namely
\begin{equation}\label{equ-BelongtoHolderianSpace-2}
\int_a^b\left(\frac{x-a}{t-a}\right)^{1-\theta/(2\pi)+\epsilon}\left(\frac{b-x}{b-t}\right)^{\epsilon}\frac{(t-a)^{\epsilon}f(t)(b-t)^{1-(\theta/(2\pi)-\mu-\epsilon)}}{t-x} dt \in H^{\epsilon/2}_0(\overline{\Omega}).
\end{equation}

\eqref{equ-BelongtoHolderianSpace-2} is guaranteed by  Property~\ref{pro:BInW} and can be justified similarly to \eqref{equ-BelongtoHolderianSpace-1} without essential difference.

Hence, $\Sigma_2\in H^*(\Omega)$, and therefore, 
\[
\frac{\psi_3(x)}{(b-x)^\mu} \in H^*(\Omega).
\]

This completes the proof for part $(3)$.

\vspace{0.2cm}

5. Proof for part $(4)$.

Using \eqref{equ-ForPsiTwoExpression} and integrating both sides by ${_aD_x^{-\mu}}$,
\begin{equation}
	{_aD_x^{-\mu}} \frac{\psi_2}{r_b^\mu}=\frac{A}{A^2+B^2}	{_aD_x^{-\mu}} \Pi_1-\frac{B}{\pi(A^2+B^2)}	{_aD_x^{-\mu}} \Pi_2.
\end{equation}

It is clear that $	{_aD_x^{-\mu}}\Pi_1  \in L^{p}(I_b)$, $p=\frac{1}{1-\frac{\theta}{2\pi}}$, $I_b=(\frac{b-a}{2}, b) $,  since $\Pi_1 =f(x)\in H^{\lambda_0}(\overline{\Omega})$. In order to show  ${_aD_x^{-\mu}} \frac{\psi_2}{r_b^\mu} \in L^{p}(I_b)$, we only need to show ${_aD_x^{-\mu}} \Pi_2\in L^{p}(I_b)$.

Recall from \eqref{equ:Pi2} that $\Pi_2\in H^*(\Omega)$, which  implies that $\Pi_2\in L^z(\Omega)$ for some $z>1$. This allows us to apply the fact (eq.  (11.17), p. 206, \cite{MR1347689}) that 
\[
{_aD_x^{-\mu}} g= \cos(\mu \pi){_xD_b^{-\mu}}g-\sin(\mu \pi) {_xD_b^{-\mu}}(r_b^{-\mu}S(r_b^{\mu}g)),\,\text{for}\, g(x)\in L^p(\Omega),p> 1
\]
to ${_aD_x^{-\mu}}\Pi_2$ to obtain
\begin{equation}
{_aD_x^{-\mu}}\Pi_2=\cos(\mu \pi){_xD_b^{-\mu}}\Pi_2-\sin(\mu \pi) {_xD_b^{-\mu}}(r_b^{-\mu}S(r_b^\mu \Pi_2)).
\end{equation}

Let us examine ${_xD_b^{-\mu}}\Pi_2$ first. Indeed,
\begin{equation}
\Pi_2\in L^{t}(I_b), \quad \text{for any}\quad t>0,
\end{equation}
by seeing that in \eqref{equ-UsedForLp} the $\epsilon$ can be chosen as small as possible.
Certainly, 
\begin{equation}\label{equ-111}
	{_xD_b^{-\mu}}\Pi_2\in L^p(I_b), \quad p=\frac{1}{1-\frac{\theta}{2\pi}}.
\end{equation}

Secondly, we look ${_xD_b^{-\mu}}(r_b^{-\mu}S(r_b^\mu \Pi_2))$.  Utilizing  the second line of  \eqref{equ-UsedForLp}, we obtain
\begin{equation}\label{equ-Analogoue}
r_b^\mu \Pi_2=\frac{\int_a^b\left(\frac{x-a}{t-a}\right)^{\epsilon}\left(\frac{b-x}{b-t}\right)^{\theta/(2\pi)+\epsilon}\frac{(t-a)^{\theta/(2\pi)+\epsilon}f(t)(b-t)^{\epsilon+\mu}}{t-x} dt}{(x-a)^{\theta/(2\pi)+\epsilon}(b-x)^{\epsilon}},
\end{equation}
where  $0<\epsilon<\text{min}\{\lambda_0, 1-\frac{\theta}{2\pi}\}$.

On the right-hand side of ~\eqref{equ-Analogoue},  the numerator belongs to $H^{\epsilon/2}_0(\overline{\Omega})$, which is justified analogously to ~\eqref{equ-BelongtoHolderianSpace-1}. Therefore, by virtue of Property~\ref{pro:BInW},
\begin{equation}\label{equ:Lemma1Pre}
r_b^{-\mu} S(r_b^\mu \Pi_2)=\frac{h(x)}{(x-a)^{\theta/(2\pi)+\epsilon}(b-x)^{\epsilon+\mu}},
\end{equation}
for a certain $h(x)\in H^{\epsilon/2}_0(\overline{\Omega})$ and $\epsilon$ is the same   as in~\eqref{equ-Analogoue}. It is straightforward to check
\[
r_b^{-\mu} S(r_b^\mu \Pi_2)\in L^{1/(\mu+2\epsilon)}(I_b).
\]

 Remembering that the $\epsilon$ can be chosen as small as possible in \eqref{equ:Lemma1Pre} and by another use of Property~\ref{pro-mappings} to \eqref{equ:Lemma1Pre} over the interval $(\frac{b-a}{2}, \, b)$, we  derive
\begin{equation}
{_xD_b^{-\mu}}(r_b^{-\mu}S(r_b^\mu \Pi_2))\in L^{\nu}(I_b),\quad \text{for any}\quad \nu >0.
\end{equation}
 In particular,
\begin{equation}\label{equ-222}
{_xD_b^{-\mu}}(r_b^{-\mu}S(r_b^\mu \Pi_2))\in L^p(I_b), \quad  p=\frac{1}{1-\frac{\theta}{2\pi}}.
\end{equation}
Combining ~\eqref{equ-111} and ~\eqref{equ-222} gives the desired result
 \begin{equation}
{_aD_x^{-\mu}} \frac{\psi_2}{r_b^\mu} \in L^{p}(I_b),
\end{equation}	
where 
\[
p=\frac{1}{1-\frac{\theta}{2\pi}},\,I_b=(\frac{b-a}{2}, b).
\]

6. Proof for part $(5)$.

Since we assume the solutions satisfy $\psi_2(x)=\psi_3(x)$, $\psi_2$ (or $\psi_3$)  belongs to $H^*(\Omega)\cap C([a,b])$ due to $\psi_2\in H^*(\Omega)\cap C( (a,b] )$ and $\psi_3\in H^*(\Omega)\cap C([a,b))$. This means the problem  \eqref{equ-RProblem-1} is solvable in the following four spaces:
\begin{equation}
H^*(\Omega), H^*(\Omega)\cap C( (a,b] ),  H^*(\Omega)\cap C([a,b)), H^*(\Omega)\cap C([a,b]).
\end{equation}

From  Property~\ref{pro:SolvabilityofDSI}, we know that, as a solution of \eqref{equ-RProblem-1}, $\psi_2$ totally has four representations in these four spaces,  namely:
	\begin{equation}
\begin{aligned}
\psi_2(x)&=\frac{AF(x)}{A^2+B^2}-\frac{B}{\pi(A^2+B^2)}\int_a^b\left( \frac{x-a}{t-a}\right)^{1-\frac{\theta}{2\pi}}\left(\frac{b-x}{b-t}\right)^{\frac{\theta}{2\pi}} \frac{F(t)}{t-x} dt\\
 &\in H^*(\Omega)\cap C([a,b]),\\
\psi_2(x)&=\frac{AF(x)}{A^2+B^2}-\frac{B}{\pi(A^2+B^2)}\int_a^b\left( \frac{x-a}{t-a}\right)^{1-\frac{\theta}{2\pi}}\left(\frac{b-x}{b-t}\right)^{\frac{\theta}{2\pi}-1} \frac{F(t)}{t-x} dt\\
 &\in H^*(\Omega)\cap C( [a,b) ),\\
\psi_2(x)&=\frac{AF(x)}{A^2+B^2}-\frac{B}{\pi(A^2+B^2)}\int_a^b\left( \frac{x-a}{t-a}\right)^{-\frac{\theta}{2\pi}}\left(\frac{b-x}{b-t}\right)^{\frac{\theta}{2\pi}} \frac{F(t)}{t-x} dt \\
&\in H^*(\Omega)\cap C((a,b]),\\
\psi_2(x)&=\frac{AF(x)}{A^2+B^2}-\frac{B}{\pi(A^2+B^2)}\int_a^b\left( \frac{x-a}{t-a}\right)^{-\frac{\theta}{2\pi}}\left(\frac{b-x}{b-t}\right)^{\frac{\theta}{2\pi}-1} \frac{F(t)}{t-x} dt\\
&+C\, (x-a)^{-\frac{\theta}{2\pi}}(b-x)^{\frac{\theta}{2\pi}-1}\\
&\in H^*(\Omega).
\end{aligned}
\end{equation}

This completes the proof  provided that $C=0$ in the forth equation. 

To see this, first we calculate
\begin{equation}
\begin{aligned}
C&=\left( \psi_2(x)-\frac{AF(x)}{A^2+B^2}\right) (x-a)^{\frac{\theta}{2\pi}}(b-x)^{1-\frac{\theta}{2\pi}}\\
&+\frac{B}{\pi(A^2+B^2)}\int_a^b\frac{(t-a)^{\theta/(2\pi)}f(t)(b-t)^{1-(\theta/(2\pi)-\mu)}}{t-x} dt,\, x\in \Omega.
\end{aligned}
\end{equation}

The first term equals to $0$ at the boundary point $x=b$.

In the second term, we can directly verify, analogously to \eqref{equ-BelongtoHolderianSpace}, that
\[
(t-a)^{\theta/(2\pi)}f(t)(b-t)^{1-(\theta/(2\pi)-\mu)}\in H^{\epsilon_0}_0(\overline{\Omega}), \epsilon_0=\text{min}\{\lambda_0,\frac{\theta}{2\pi}, 1-(\frac{\theta}{2\pi}-\mu)\}.
\]
Thus, in light of  Property \ref{pro:BInW}, we see
\[
\frac{1}{\pi}\int_a^b\frac{(t-a)^{\theta/(2\pi)}f(t)(b-t)^{1-(\theta/(2\pi)-\mu)}}{t-x} dt \in H^{\epsilon_0}_0(\overline{\Omega}).
\]

Consequently, 
\begin{equation}
\begin{aligned}
C&=\lim_{x\rightarrow b^-} \left( \psi_2(x)-\frac{AF(x)}{A^2+B^2}\right) (x-a)^{\frac{\theta}{2\pi}}(b-x)^{1-\frac{\theta}{2\pi}}\\
&+\lim_{x\rightarrow b^-} \frac{B}{\pi(A^2+B^2)}\int_a^b\frac{(t-a)^{\theta/(2\pi)}f(t)(b-t)^{1-(\theta/(2\pi)-\mu)}}{t-x} dt\\
&=0+0\\
&=0.
\end{aligned}
\end{equation}

This completes the proof of part $(5)$, and the whole proof of Lemma \ref{lemma-Pre-RaiseRegularity} is completed.
\end{proof}
The following lemma provides a bridge connecting the solutions of   coupled Abel integral equations in the Sobolev space to the solutions  in the space $H^*(\Omega)$. By which we mean that, if a solution $\psi(x)$ of the coupled Abel integral equation is located in  $\widehat{H}^{(1+\mu)/2}_0(\Omega)$ (see equation \eqref{eq:RaisingRegularity}), then its $\mu$-th order derivative ${_aD_x^\mu}\psi$ actually has a representative belonging to $H^*(\Omega)$ (which is equivalent to what equation \eqref{equ-Representation of J-1} says).
 This is an important connection and preparation for us to  continue to raise the regularity of ${_aD_x^\mu}\psi$ to better spaces $H^*_\sigma(\Omega)$ from $H^*(\Omega)$ in  subsequent Lemma \ref{lem: RaisingTheRegularity}, thereby raising the regularity of $\psi(x)$.
 
 As we will see later, this $\psi$ in \eqref{eq:RaisingRegularity}  actually represents the weak solution of our problem \eqref{equationformainresult}; Lemma \ref{lem-raising ragularity-1} and   \ref{lem: RaisingTheRegularity} are two key intermediate steps towards converting the weak solution to the classical solution.
\begin{lemma}\label{lem-raising ragularity-1}
Let $c$ be a constant, $\psi(x)\in \widehat{H}^{(1+\mu)/2}_0(\Omega)$ and $f(x)\in H(\overline{\Omega}) $. If 
\begin{equation}\label{eq:RaisingRegularity}
\alpha _aD_x^{-(1-\mu)}\psi+\beta {_xD_b^{-(1-\mu)}}\psi \overset{a.e.}{=}{_aD_x^{-1}}f+c, \; x\in \Omega, 
\end{equation}
then the solution $\psi(x)$ has a representation
\begin{equation}\label{equ-Representation of J-1}
	\psi(x)={_aD_x^{-\mu}}J,
\end{equation}
where $J(x)\in H^*(\Omega)$ and $J(x)$ has four equivalent representations:
\begin{equation}\label{equ-Representation of J-2}
	\begin{aligned}
	J(x)&=\frac{A f(x)}{A^2+B^2}-\frac{B}{\pi(A^2+B^2)}\int_a^b\left( \frac{x-a}{t-a}\right)^{1-\frac{\theta}{2\pi}}\left(\frac{b-x}{b-t}\right)^{\frac{\theta}{2\pi}-\mu} \frac{ f(t)}{t-x} dt,\\
		J(x)&=\frac{A f(x)}{A^2+B^2}-\frac{B}{\pi(A^2+B^2)}\int_a^b\left( \frac{x-a}{t-a}\right)^{1-\frac{\theta}{2\pi}}\left(\frac{b-x}{b-t}\right)^{\frac{\theta}{2\pi}-\mu-1} \frac{ f(t)}{t-x} dt,\\
			J(x)&=\frac{A f(x)}{A^2+B^2}-\frac{B}{\pi(A^2+B^2)}\int_a^b\left( \frac{x-a}{t-a}\right)^{-\frac{\theta}{2\pi}}\left(\frac{b-x}{b-t}\right)^{\frac{\theta}{2\pi}-\mu} \frac{ f(t)}{t-x} dt,\\
				J(x)&=\frac{A f(x)}{A^2+B^2}-\frac{B}{\pi(A^2+B^2)}\int_a^b\left( \frac{x-a}{t-a}\right)^{-\frac{\theta}{2\pi}}\left(\frac{b-x}{b-t}\right)^{\frac{\theta}{2\pi}-\mu-1} \frac{ f(t)}{t-x} dt.
	\end{aligned}
\end{equation}
where 	$A=\alpha-\beta\cos(\mu\pi)$, $B=\beta\sin(\mu\pi)$.
\end{lemma}

\begin{proof}\;
1.  Differentiating both sides of equation~\eqref{eq:RaisingRegularity} is valid by the assumption $\psi\in \widehat{H}_0^{(1+\mu)/2}(\Omega)$ and Property~\ref{pro:alternate-form}, from which we have
\begin{equation}
D(\alpha _aD_x^{-(1-\mu)}\psi+\beta {_xD_b^{-(1-\mu)}}\psi )\overset{a.e.}{=}f.
\end{equation}
Distributing the differentiation operator $D$ is permitted and gives
\begin{equation}\label{equ:Raising-1}
\alpha {_aD_x^\mu}\psi-\beta {_xD_b^\mu}\psi\overset{a.e.}{=}f.
\end{equation}

Noting $(1+\mu)/2>\mu$ and recalling the knowledge of embedding $\widehat{H}^{(1+\mu)/2}_0(\Omega)$   $\subset$   $ \widehat{H}^{\mu}_0(\Omega)$, we see  from Property~\ref{pro:alternate-form}  that  $\psi(x)$ can be represented as $\psi(x)={_aD_x^{-\mu}}{_aD_x^\mu}\psi$ and ${_aD_x^\mu}\psi\in L^2(\Omega)$.   

Doing back substitution for $\psi$ into \eqref{equ:Raising-1}, the left-hand side becomes
\begin{equation}\label{equ-SubstitutingLeftIngegral}
 \alpha {_aD_x^\mu}\psi-\beta {_xD_b^\mu}\psi=\alpha {_aD_x^\mu}\psi-\beta {_xD_b^\mu}{_aD_x^{-\mu}}{_aD_x^\mu}\psi.
\end{equation}

For notation simplicity, we denote $r_a(x)=x-a, r_b(x)=b-x,  x\in \overline{\Omega}$.

Applying the fact (eq. (11.17), p. 206, \cite{MR1347689}) that 
\[
	{_aD_x^{-\mu}} g= \cos(\mu \pi){_xD_b^{-\mu}}g-\sin(\mu \pi) {_xD_b^{-\mu}}(r_b^{-\mu}S(r_b^{\mu}g)),\,\forall g(x)\in L^p(\Omega),p> 1
\]
 to the right-hand side of \eqref{equ-SubstitutingLeftIngegral}, we have
\begin{equation}
\begin{aligned}
&=\alpha {_aD_x^\mu}\psi-\beta \cos(\mu\pi){_aD_x^\mu}\psi+\beta\sin(\mu\pi)r_b^{-\mu}S(r_b^\mu{_aD_x^\mu}\psi)\\
&=\left( \alpha-\beta\cos(\mu\pi)\right){_aD_x^\mu}\psi+\beta\sin(\mu\pi)r_b^{-\mu}S(r_b^\mu{_aD_x^\mu}\psi).
\end{aligned}
\end{equation}
Inserting this back into equation~\eqref{equ:Raising-1},  multiplying both sides by $r_b^\mu$ and denoting \[
A=\alpha-\beta\cos(\mu\pi),\, B=\beta\sin(\mu\pi),\, \Psi_1(x)=r_b^\mu{_aD_x^\mu}\psi,\, F(x)=r_b^\mu f(x), 
\]
we arrive at
\begin{equation}\label{equ:SDI}
A\Psi_1(x) + \frac{B}{\pi}\int_a^b\frac{\Psi_1(t)}{t-x} dt \overset{a.e.}{=}F(x), x\in \Omega.
\end{equation}

\vspace{0.2cm}

2.  On the other hand, by Lemma~\ref{lemma-Pre-RaiseRegularity}, we already know that there exist solutions $\Psi_2 \in H^*(\Omega)\cap C((a, b])$ and $\Psi_3\in H^*(\Omega)\cap C([a, b))$ satisfying
\begin{equation}\label{equ:SDI2} 
A\Psi_i(x)+ \frac{B}{\pi}\int_a^b\frac{\Psi_i(x)}{t-x} dt =F(x), x\in \Omega, (i=2,3),
\end{equation}
and that
\begin{equation}\label{equ-BelongtoH^*}
\frac{\Psi_2}{(b-x)^\mu}, \frac{\Psi_3}{(b-x)^\mu} \in H^*(\Omega), \quad \mu<\frac{\theta}{2\pi}<1.
\end{equation}

Notice that the distinction between \eqref{equ:SDI} and \eqref{equ:SDI2} is that \eqref{equ:SDI} holds $a.e.$ and \eqref{equ:SDI2} holds for every $x\in \Omega$ and that $\Psi_2$ and $\Psi_3$ belong to $H^*(\Omega)$ but $\Psi_1$ is not immediately clear yet for now. 
 
In the following, our strategy is to intend to show that
\begin{equation}
\Psi_1(x)=\Psi_2(x)=\Psi_3(x),
\end{equation}
which  produces \eqref{equ-Representation of J-1} in the lemma, and after this, \eqref{equ-Representation of J-2} will be obtained shortly, the whole lemma hence will be eventually completed.

\vspace{0.2cm}

3. Now let us continue. 

Subtracting \eqref{equ:SDI2} from \eqref{equ:SDI} gives
\begin{equation}\label{equ:SDI23}
A\Psi(x) + \frac{B}{\pi}\int_a^b\frac{\Psi(t)}{t-x} dt\overset{a.e.}{=}0,\, x\in \Omega,
\end{equation}
where 
\[
\Psi(x)=\Psi_1(x)-\Psi_i(x), (i=2,3).
\]

(Evidently, $\Psi(x)$ depends on $i$ and it should not be confused. We denote $\Psi(x)$ this way simply because it is more convenient to discuss this two cases $i=2,3$ together in the rest of proof rather than separately.)

Dividing both sides of \eqref{equ:SDI23} by $r_b^\mu$, we arrive at
\begin{equation}\label{equ:SDI234}
A\widetilde{\Psi}(x) + \frac{B}{\pi}\frac{1}{(b-x)^\mu}\int_a^b\frac{(b-t)^\mu \widetilde{\Psi}(t)}{t-x} dt\overset{a.e.}{=}0, \, x\in \Omega,
\end{equation}
where 
\[
\widetilde{\Psi}(x)=\frac{\Psi(x)}{(b-x)^\mu}=\frac{\Psi_1(x)}{(b-x)^\mu}-\frac{\Psi_i(x)}{(b-x)^\mu}, (i=2, 3).
\]

\vspace{0.2cm}

4. Let us examine which functional space $\widetilde{\Psi}(x)$ belongs to  (discussing the two cases $i=2,3$ together).

First,
\[
\frac{\Psi_1(x)}{(b-x)^\mu}={_aD_x^\mu}\psi \in L^2(\Omega), \, \text{since $\psi\in \widehat{H}^{(1+\mu)/2}_0(\Omega)$}.
\]
Secondly, since 
\begin{equation}
	\frac{\Psi_i}{(b-x)^\mu}\in H^*(\Omega),
\end{equation}
we can always find a certain $p>1$ that does not depend on $i$ such that
\begin{equation}
		\frac{\Psi_i}{(b-x)^\mu}\in L^p(\Omega), (i=2,3),
\end{equation}
by taking into account the definition of  $H^*(\Omega)$.

 Combining these together, we  conclude that
\begin{equation}\label{eq-find a space}
\widetilde{\Psi}(x) \in L^p(\Omega),\, \text{for a certain $p>1$}.
\end{equation}

\vspace{0.2cm}

5. The establishment of  \eqref{eq-find a space}  allows us to be  able to  apply Property~\ref{pro:CR} to equation~\eqref{equ:SDI234}. To do so, integrating both sides of ~\eqref{equ:SDI234} by ${_xD_b^{-\mu}}$, we obtain
\begin{equation}
A{_xD_b^{-\mu}}\widetilde{\Psi} + \frac{B}{\pi}(x-a)^{\mu}\int_a^b\frac{{_tD_b^{-\mu}}\widetilde{\Psi} }{(t-a)^\mu(t-x)} dt\overset{a.e.}{=}0, \, x\in \Omega.
\end{equation}
Dividing both sides by $(x-a)^\mu$, we arrive at
\begin{equation}\label{equ:SDI2345}
A\widetilde{\widetilde{\Psi}}(x) + \frac{B}{\pi}\int_a^b\frac{\widetilde{\widetilde{\Psi}}(t)}{t-x} dt\overset{a.e.}{=}0, \, x\in \Omega,
\end{equation}
where
\[
\widetilde{\widetilde{\Psi}}(x)=\frac{{_xD_b^{-\mu}}\widetilde{\Psi}}{(x-a)^\mu}.
\]

Before going further, let us call attention to that \eqref{equ:SDI2345} holds $a.e.$ at this stage since $\widetilde{\widetilde{\Psi}}(x)$ is essentially  in terms of $\psi(x)$, and in the next two steps we intend to show that $\widetilde{\widetilde{\Psi}}(x)$ actually admits a good representative such that \eqref{equ:SDI2345} holds for every $x\in \Omega$.

\vspace{0.2cm}

6. Now we assert that $\widetilde{\widetilde{\Psi}}(x)\in H^*(\Omega)$ (for both $i=2,3$). (It should be clear that $\widetilde{\widetilde{\Psi}}(x)\in H^*(\Omega)$ means there is a representative in the equivalence classes of $\widetilde{\widetilde{\Psi}}(x)$ such that it belongs to $ H^*(\Omega)$).

 To see this, substituting for $\widetilde{\Psi}$ into $\widetilde{\widetilde{\Psi}}(x)$, we  have
\begin{equation}\label{equ-Doulble tidle}
\begin{aligned}
\widetilde{\widetilde{\Psi}}(x)&=\frac{1}{(x-a)^\mu} {_xD_b^{-\mu}}\left(\frac{1}{r_b^\mu}\Psi_1-\frac{1}{r_b^\mu}\Psi_i \right)\\
&=\frac{1}{(x-a)^\mu}{_xD_b^{-\mu}}{_aD_x^\mu}\psi -\frac{1}{(x-a)^\mu}{_xD_b^{-\mu}}\frac{\Psi_i}{r_b^\mu}\\
&=M_1-M_2.
\end{aligned}
\end{equation}

It suffices to investigate $M_1$ and $M_2$, respectively.

Consider $M_1$ first and examine  the piece ${_xD_b^{-\mu}}{_aD_x^\mu}\psi$. By Property~\ref{pro:alternate-form}, there exists a function $\psi_1(x)\in L^2(\Omega)$ such that 
\[
{_aD_x^\mu}\psi={_aD_x^\mu}{_aD_x^{-(1+\mu)/2}}\psi_1={_aD_x^{-(1-\mu)/2}}\psi_1.
\]
On the other hand,   there exists a  function $\psi_2(x)\in L^2(\Omega)$ such that
\[
{_aD_x^{-(1-\mu)/2}}\psi_1={_xD_b^{-(1-\mu)/2}}\psi_2,
\]
 (Corollary 1, p. 208, \cite{MR1347689}). 
Thus,
\[
{_xD_b^{-\mu}}{_aD_x^\mu}\psi={_xD_b^{-(1+\mu)/2}}\psi_2\in  H^{\mu/2}  (\overline{\Omega})
\]
 is guaranteed by Property \ref{pro:MappingIntoHolderianFromLp}.
 
 If we equivalently rewrite $M_1$ as
 \begin{equation}\label{eq-Rerepresentation}
 M_1=\frac{(x-a)^\epsilon ({_xD_b^{-\mu}}{_aD_x^\mu}\psi) (b-x)^\epsilon}{(x-a)^{1-(1-\mu-\epsilon)}(b-x)^{1-(1-\epsilon)}},
 \end{equation}
 where $\epsilon$ is chosen so that $0<\epsilon<\min\{1-\mu,\mu/2\}$, then by simple steps as we justified for $\Pi_1$ in the step 3 of  the proof of Lemma~\ref{lemma-Pre-RaiseRegularity},
 we see
 \[
 M_1\in H^*(\Omega),
 \]
 and not repeated here.
 
For  $M_2$, indeed, by virtue of Property~\ref{pro-correspondence-1} and recalling \eqref{equ-BelongtoH^*}, ${_xD_b^{-\mu}}\frac{\Psi_i}{r_b^\mu}$ can be represented as
\begin{equation}
{_xD_b^{-\mu}}\frac{\Psi_i}{r_b^\mu}=\frac{g_i(x)}{(x-a)^{1-\epsilon_1}(b-x)^{1-\epsilon_2}},
\end{equation}
for certain functions $g_i(x)$ and real numbers $k,\epsilon_1,\epsilon_2$ satisfying
\begin{equation}
g_i(x)\in H^k_0(\overline{\Omega}),\quad \mu<k, 0<\epsilon_1, 0<\epsilon_2 \quad (\text{$k,\epsilon_1,\epsilon_2$ depend on $i$}).
\end{equation}
Taking into account the useful fact that
\begin{equation}
\frac{g_i(x)}{(x-a)^\mu}\in H^l_0(\overline{\Omega}) , \quad \text{for any $0<l<k-\mu$},
\end{equation}
 (the value of $\frac{g_i(x)}{(x-a)^\mu}$ at $x=a$ is understood in the  limiting sense) and the definition of  $H^*(\Omega)$, we see

\begin{equation}
M_2=\frac{1}{(x-a)^\mu}{_xD_b^{-\mu}}\frac{\Psi_i}{r_b^\mu}=\frac{\frac{g_i(x)}{(x-a)^\mu} }{(x-a)^{1-\epsilon_1}(b-x)^{1-\epsilon_2}}\in H^*(\Omega), (i=2,3).
\end{equation}

Combing $M_1$ and $M_2$ yields the assertion
\[
\widetilde{\widetilde{\Psi}}(x)\in H^*(\Omega)\quad  \text{for both}\, i=2,3.
\]

\vspace{0.2cm}

7. Once we have the above, it now is notable that equation~\eqref{equ:SDI2345} becomes valid for every point in $\Omega$, namely
\begin{equation}\label{lemma2dominantSI}
A\widetilde{\widetilde{\Psi}}(x) + \frac{B}{\pi}\int_a^b\frac{\widetilde{\widetilde{\Psi}}(t)}{t-x} dt=0,\, \text{for each $x\in \Omega$}.
\end{equation}

(Putting it another way, there is a representative for the equivalence classes of $\widetilde{\widetilde{\Psi}}(x) $ such that it belongs to $H^*(\Omega)$ and makes \eqref{equ:SDI2345} hold for every $x\in \Omega$ ).

On the other hand, remember that, in $H^*(\Omega)$, \eqref{lemma2dominantSI} is unconditionally solvable according to the part $(1)$ of Property~\ref{pro:SolvabilityofDSI}. Hence by utilizing \eqref{eq-sloution-CSE}, we derive that  $\widetilde{\widetilde{\Psi}}(x)$ can be represented as a constant multiple of $(x-a)^{1-n_a-\frac{\theta}{2\pi}}(b-x)^{\frac{\theta}{2\pi}-n_b}$ with choosing $n_a=1,n_b=1$, namely
\begin{equation}\label{equ:SI}
\widetilde{\widetilde{\Psi}}(x)=C_i \cdot(x-a)^{-\frac{\theta}{2\pi}}(b-x)^{\frac{\theta}{2\pi}-1}, 
\end{equation}
where $C_i$ depends on the $i$ in $\widetilde{\widetilde{\Psi}}(x)$  (i=2,3).

In the next two steps, we will show $C_2$ and $C_3$ have to be zero, separately.

\vspace{0.2cm}

8. Using the second line of expression~\eqref{equ-Doulble tidle}, we solve equation~\eqref{equ:SI} for ${_aD_x^\mu}\psi$ to obtain
\begin{equation}\label{eq-Final Equation}
{_aD_x^\mu}\psi=\frac{\Psi_i(x)}{(b-x)^\mu}+C_i \,{_xD_b^\mu}((x-a)^{-\frac{\theta}{2\pi}+\mu}(b-x)^{\frac{\theta}{2\pi}-1}),\, (i=2,3).
\end{equation}

Integrating both sides by ${_aD_x^{-\mu}}$, which is valid, and also noting $\psi(x)={_aD_x^{-\mu}}{_aD_x^{\mu}}\psi$, we have
\begin{equation}\label{equ:Psifunctiontosimplify}
\psi(x)={_aD_x^{-\mu}}\frac{\Psi_i}{r_b^\mu}+C_i \,{_aD_x^{-\mu}}{_xD_b^\mu}((x-a)^{-\frac{\theta}{2\pi}+\mu}(b-x)^{\frac{\theta}{2\pi}-1}),\, (i=2,3).
\end{equation}

Calculating the second term on the right-hand side (using eq. (11.4) and (11.19), \cite{MR1347689}),
\begin{equation}
\begin{aligned}
&{_aD_x^{-\mu}}{_xD_b^\mu}((x-a)^{-\frac{\theta}{2\pi}+\mu}(b-x)^{\frac{\theta}{2\pi}-1})\\
&=\left(\cos(\mu \pi)+\sin(\mu \pi)\cot(\pi-\frac{\theta}{2}) \right)(x-a)^{-\frac{\theta}{2\pi}+\mu}(b-x)^{\frac{\theta}{2\pi}-1}\\
&=\widetilde{C} \cdot (x-a)^{-\frac{\theta}{2\pi}+\mu}(b-x)^{\frac{\theta}{2\pi}-1},
\end{aligned}
\end{equation}
where
\[
\widetilde{C} =\cos(\mu \pi)+\sin(\mu \pi)\cot(\pi-\frac{\theta}{2}).
\]
 Equation \eqref{equ:Psifunctiontosimplify} further becomes
\begin{equation}\label{constantshavetobezero}
\psi(x)={_aD_x^{-\mu}}\frac{\Psi_i}{r_b^\mu}+C_i\widetilde{C} \cdot (x-a)^{-\frac{\theta}{2\pi}+\mu}(b-x)^{\frac{\theta}{2\pi}-1}, \, (i=2,3).
\end{equation}

Consider the case $i=3$ now. On one hand, recall that $\mu<\frac{\theta}{2\pi}<1$ from \eqref{equ-BelongtoH^*}. This implies that the coefficient $\widetilde{C}$ is non-zero and hence that $\widetilde{C} \cdot(x-a)^{-\frac{\theta}{2\pi}+\mu}(b-x)^{\frac{\theta}{2\pi}-1}$ is unbounded at the boundary point $x=a$. On the other hand,   both $\psi(x)$ (belonging to $\widehat{H}^{(1+\mu)/2}_0(\Omega)$) and ${_aD_x^{-\mu}}\frac{\Psi_3}{r_b^\mu}$  (remember $\Psi_3\in H^*(\Omega)\cap C([a, b))$ ) are bounded at $x=a$, which contradicts the unboundedness of $C_3\widetilde{C} \cdot (x-a)^{-\frac{\theta}{2\pi}+\mu}(b-x)^{\frac{\theta}{2\pi}-1}$ unless $C_3=0$. Thereby, we  arrive at
\begin{equation}\label{equ-first-equality-for-i=3}
\psi(x)={_aD_x^{-\mu}}\frac{\Psi_3}{r_b^\mu}.
\end{equation}

So, we have proved \eqref{equ-Representation of J-1} in the lemma by letting $J(x)=\frac{\Psi_3}{r_b^\mu}$ and noting $J(x) \in H^*(\Omega)$ from \eqref{equ-BelongtoH^*}. We remain to show that $\frac{\Psi_3}{r_b^\mu}$ has four equivalent representations, namely \eqref{equ-Representation of J-2}, which follows from the next step.

\vspace{0.2cm}

9.   In this last step we  show that the constant $C_2$ in \eqref{constantshavetobezero} for the case $i=2$ has to be zero as well, namely, $C_2=0$.

Using equation~\eqref{constantshavetobezero} and doing subtraction with each other for $i=2,3$, 
\begin{equation}\label{equ:DoingSubtracting}
{_aD_x^{-\mu}}\frac{\Psi_3}{r_b^\mu}-{_aD_x^{-\mu}}\frac{\Psi_2}{r_b^\mu}=C_2\widetilde{C}\cdot(x-a)^{-\frac{\theta}{2\pi}+\mu}(b-x)^{\frac{\theta}{2\pi}-1}.
\end{equation}

Let us take care about each term  above  at the  boundary point $x=b$ to  obtain a contradiction.

For the left-hand side, using  \eqref{equ-first-equality-for-i=3} and invoking the part $(4)$ of Lemma~\ref{lemma-Pre-RaiseRegularity}, we know
\begin{equation}
{_aD_x^{-\mu}}\frac{\Psi_3}{r_b^\mu} \left(=\psi(x)\in \widehat{H}^{(1+\mu)/2}_0(\Omega)\right), \,\, {_aD_x^{-\mu}}\frac{\Psi_2}{r_b^\mu}  \in L^{p}(I_b),
\end{equation}
where $p=\frac{1}{1-\frac{\theta}{2\pi}},\,I_b=(\frac{b-a}{2}, b) $. However, in the right-hand side of \eqref{equ:DoingSubtracting}, it is clear that
\begin{equation}
\widetilde{C}\cdot(x-a)^{-\frac{\theta}{2\pi}+\mu}(b-x)^{\frac{\theta}{2\pi}-1}  \notin L^{p}(I_b).
\end{equation}

This means that $C_2$ has to be zero in \eqref{equ:DoingSubtracting}, from which it follows that
\begin{equation}\label{equ-first-equality-for-i=2}
\psi(x)={_aD_x^{-\mu}}\frac{\Psi_2}{r_b^\mu}.
\end{equation}

Comparing \eqref{equ-first-equality-for-i=3} and \eqref{equ-first-equality-for-i=2},  we see 
\begin{equation}
{_aD_x^{-\mu}}(\frac{\Psi_3}{r_b^\mu}-\frac{\Psi_2}{r_b^\mu})=0.
\end{equation}

Only trivial solution is allowed by  seeing $\frac{\Psi_2}{r_b^\mu}, \frac{\Psi_3}{r_b^\mu} \in H^*(\Omega)$ from \eqref{equ-BelongtoH^*}, and thus
\begin{equation}\label{equ:equality}
\Psi_2(x)=\Psi_3(x),\, x\in \Omega.
\end{equation}

Once we have  equality \eqref{equ:equality},  by virtue of the part (5) of Lemma~\ref{lemma-Pre-RaiseRegularity}, \eqref{equ-Representation of J-2}, namely the four desired representations for $\frac{\Psi_3}{r_b^\mu}$ (or $\frac{\Psi_2}{r_b^\mu}$) in the lemma,  follows immediately  after dividing by $r_b^\mu$.

This finally completes the whole proof.
\end{proof}
Now we go one  step further from above lemma by showing that ${_aD_x^\mu}\psi$  can actually go  to better spaces $H_\sigma^*(\Omega)$ from  $H^*(\Omega)$ provided that $f$ lies in $H_\sigma^*(\Omega)$ (which is the same as what \eqref{equ:lemma3representation} means in the following).

\begin{lemma}\label{lem: RaisingTheRegularity}
	 Given  $0<\sigma<1$,  let $c$ be a constant, $\psi(x)\in \widehat{H}^{(1+\mu)/2}_0(\Omega)$ and $f(x)\in H(\overline{\Omega}) \cap H^*_\sigma(\Omega)$. If 
	\begin{equation}\label{equ-GoFurther}
	\alpha _aD_x^{-(1-\mu)}\psi+\beta {_xD_b^{-(1-\mu)}}\psi \overset{a.e.}{=}{_aD_x^{-1}}f +c, x\in\Omega,
	\end{equation}
	then the solution $\psi(x)$ has a representation
	\begin{equation}\label{equ:lemma3representation}
	\psi(x)={_aD_x^{-(\mu+\sigma)}}K_\sigma,
	\end{equation}
	where function $K_\sigma(x)$ belongs to $H^*(\Omega)$ ($K_\sigma$ depends on $\sigma$).
\end{lemma}	
\begin{proof}  Notice that, compared to  Lemma~\ref{lem-raising ragularity-1}, only one more condition is imposed in this lemma, namely $f(x)\in H^*_\sigma(\Omega), 0<\sigma<1$. 

\vspace{0.2cm}

		1. As a direct consequence of Lemma~\ref{lem-raising ragularity-1}, we  already know that the function $\psi(x)$ must have a representation
		\begin{equation}\label{equ-FirstIntegralEx}
		\psi(x)={_aD_x^{-\mu}}J,
		\end{equation}
		where $J(x)\in H^*(\Omega)$ and $J(x)$ has four equivalent representations:
		\begin{equation}\label{equ-FourDiffRepresentations}
		\begin{aligned}
		J(x)&=\frac{A f(x)}{A^2+B^2}-\frac{B}{\pi(A^2+B^2)}\int_a^b\left( \frac{x-a}{t-a}\right)^{1-\frac{\theta}{2\pi}}\left(\frac{b-x}{b-t}\right)^{\frac{\theta}{2\pi}-\mu} \frac{ f(t)}{t-x} dt;\\
		J(x)&=\frac{A f(x)}{A^2+B^2}-\frac{B}{\pi(A^2+B^2)}\int_a^b\left( \frac{x-a}{t-a}\right)^{1-\frac{\theta}{2\pi}}\left(\frac{b-x}{b-t}\right)^{\frac{\theta}{2\pi}-\mu-1} \frac{ f(t)}{t-x} dt;\\
		J(x)&=\frac{A f(x)}{A^2+B^2}-\frac{B}{\pi(A^2+B^2)}\int_a^b\left( \frac{x-a}{t-a}\right)^{-\frac{\theta}{2\pi}}\left(\frac{b-x}{b-t}\right)^{\frac{\theta}{2\pi}-\mu} \frac{ f(t)}{t-x} dt;\\
		J(x)&=\frac{A f(x)}{A^2+B^2}-\frac{B}{\pi(A^2+B^2)}\int_a^b\left( \frac{x-a}{t-a}\right)^{-\frac{\theta}{2\pi}}\left(\frac{b-x}{b-t}\right)^{\frac{\theta}{2\pi}-\mu-1} \frac{ f(t)}{t-x} dt.
		\end{aligned}
		\end{equation}
		where 	$A=\alpha-\beta\cos(\mu\pi)$, $B=\beta\sin(\mu\pi)$.

For ease of notation, we denote \eqref{equ-FourDiffRepresentations} as 
	\begin{equation}
\begin{aligned}
J(x)&=\text{Expression}\&1;\\
J(x)&=\text{Expression}\&2;\\
J(x)&=\text{Expression}\&3;\\
J(x)&=\text{Expression}\&4.
\end{aligned}
\end{equation}
	
2. For any given $0<\sigma<1$, $\sigma$  must satisfy one of the following inequality cases:
	\begin{equation}
\begin{aligned}
\text{Case}\&1&: \sigma+\frac{\theta}{2\pi}\geq 1, \sigma +\mu\geq \frac{\theta}{2\pi};\\
\text{Case}\&2&: \sigma+\frac{\theta}{2\pi}\geq 1, \sigma +\mu< \frac{\theta}{2\pi};\\
\text{Case}\&3&: \sigma+\frac{\theta}{2\pi}< 1, \sigma +\mu\geq \frac{\theta}{2\pi};\\
\text{Case}\&4&: \sigma+\frac{\theta}{2\pi}< 1, \sigma +\mu<\frac{\theta}{2\pi}.\\
\end{aligned}
\end{equation}

3. We associate different cases to different expressions as follows:
	\begin{equation}\label{equ-MixedTogether}
\begin{aligned}
\text{Case}\&1& \quad \text{with} \quad\text{Expression}\&1;\\
\text{Case}\&2&\quad\text{with} \quad \text{Expression}\&2;\\
\text{Case}\&3&\quad\text{with} \quad\text{Expression}\&3;\\
\text{Case}\&4&\quad\text{with} \quad \text{Expression}\&4.\\
\end{aligned}
\end{equation}

For each of \eqref{equ-MixedTogether}, applying Property~\ref{pro-WeightedMapping} to the according expression of $J(x)$ is valid by checking the two inequality conditions in Property~\ref{pro-WeightedMapping} and  yields that
\begin{equation}
	J(x)\in H^*_\sigma(\Omega).
\end{equation}

4. Utilizing Property~\ref{pro-correspondence-1}, we know that there exists a  certain function $K_\sigma(x)\in H^*(\Omega)$ such that 
\begin{equation}\label{equ:Tobeinserted}
J(x)={_aD_x^{-\sigma}}K_\sigma,
\end{equation}
	and therefore, by inserting \eqref{equ:Tobeinserted} back into \eqref{equ-FirstIntegralEx},
		\begin{equation}
	\psi(x)={_aD_x^{-(\mu+\sigma)}}K_\sigma
	\end{equation}
	follows from the semigroup property of R-L integral operators. 
	
	This completes the whole proof.
\end{proof}

\section{Proof of Theorem  \ref{theorem}}\label{sec:proofoftheorem}

 We intend to show that the weak solution of problem \eqref{equationformainresult} associated with conditions \eqref{conditionforthemainresult} is actually the true solution to \eqref{equationformainresult} by picking up  regularity and inverting the variational formulation back to original problem \eqref{equationformainresult} pointwisely. The whole proof will be completed by invoking the lemmas that were established in previous sections; the majority of  proof is about the  existence of $u(x)$ satisfying  $u(x)= {_aD_x^{-t}J_t}$, $t<1+\mu$,  from which the confirmation of  true solution  follows.  The uniqueness will be proved at the very end.

\begin{proof}
 1.  For symbol convenience,  denote $s=(1+\mu)/2$ throughout the proof.  From Lemma~\ref{lem:Existence-weak}, we know that there exists a  unique $u\in \widehat{H}^s_0(\Omega)$ such that 
\begin{equation}\label{the:equ-1}
B_2[u, \psi]=(\frac{f}{k}, \psi)_\Omega,
\end{equation}
 for any $\psi\in C_0^\infty(\Omega)$. 
 
 Utilizing the expression of $B_2[\cdot,\cdot]$ in Definition \ref{def-bilinear}, simplifying and operating both sides of \eqref{the:equ-1}, we obtain
 \[
  (L_1(u), D\psi)_\Omega=({_xD_b^{-1}}\frac{f}{k}, D\psi)_\Omega,
 \]
  where 
 \begin{equation}
 \begin{aligned}
 L_1(u)&=\alpha {_aD_x^{\mu}}u-\beta{_xD_b^{\mu}u}\\
 &-\alpha{_xD_b^{-1}(\frac{k'}{k}{_aD_x^{\mu}}u)}
 +\beta {_xD_b^{-1}(\frac{k'}{k}{_xD_b^\mu}u)}\\
 &-\frac{p}{k}u
 -{_xD_b^{-1}}((\frac{p}{k})'u)+{_xD_b^{-1}(\frac{q}{k}u)}.
 \end{aligned}
 \end{equation}
 Namely,
 \begin{equation}\label{equ:variationL}
  (L_1(u)-{_xD_b^{-1}}\frac{f}{k}, D\psi)_\Omega=0, \forall \psi\in C_0^\infty(\Omega).
 \end{equation}

2. Observe that $L_1(u)$ is a summable function, and so is ${_xD_b^{-1}}\frac{f}{k}$. From \eqref{equ:variationL}, there exists a constant $C_1$ such that 
 \begin{equation}
 L_1(u)-{_xD_b^{-1}}\frac{f}{k}\overset{a.e.}{=}C_1, x\in \Omega. 
 \end{equation}
 Isolating the first two terms of  $L_1(u)$ gives
 \begin{equation}\label{thm:Weak-Almost}
 \alpha {_aD_x^{\mu}}u-\beta{_xD_b^{\mu}u}\overset{a.e.}{=}L_2(u)+\frac{p}{k}u,
 \end{equation}
 where 
 \begin{equation}\label{equ:L2Expression}
 \begin{aligned}
 L_2(u)&=\alpha{_xD_b^{-1}(\frac{k'}{k}{_aD_x^{\mu}}u)}
 -\beta {_xD_b^{-1}(\frac{k'}{k}{_xD_b^\mu}u)}\\
 &
 +{_xD_b^{-1}}((\frac{p}{k})'u)-{_xD_b^{-1}(\frac{q}{k}u)}+{_xD_b^{-1}}\frac{f}{k}+C_1.
 \end{aligned}
 \end{equation}
Namely,
 \begin{equation}\label{equ:taking-integ}
 D(\alpha {_aD_x^{-(1-\mu)}}u+\beta{_xD_b^{-(1-\mu)}u})\overset{a.e.}{=}L_2(u)+\frac{p}{k}u.
 \end{equation}
 
 \vspace{0.2cm}
 
 3. In order to get rid of the differentiation operator $D$ in the left-hand side above, we would like to integrate  both sides by ${_aD_x^{-1}}$, which is, however, in general valid for absolutely continuous functions on $\overline{\Omega}$. 
 
 To see $\alpha {_aD_x^{-(1-\mu)}}u+\beta{_xD_b^{-(1-\mu)}u}\in AC(\overline{\Omega})$ (it should be clear that it means there exists a representative of equivalence classes of $\alpha {_aD_x^{-(1-\mu)}}u+\beta{_xD_b^{-(1-\mu)}u}$ belonging to $AC(\overline{\Omega})$, similarly in the following steps), notice that $u\in \widehat{H}^s_0(\Omega)$, which implies that (Property~\ref{pro:alternate-form})
 \begin{equation}\label{expression of u}
 u={_aD^{-s}_x \theta_1}={_xD_b^{-s}\theta_2},
 \end{equation}
 for  certain functions $\theta_1,\theta_2\in L^2(\Omega)$. It follows that
 \begin{equation}
 \alpha {_aD_x^{-(1-\mu)}}u={_aD_x^{-1}}\tilde{\theta}_1 \, \text{and}~ \beta{_xD_b^{-(1-\mu)}}u={_xD_b^{-1}}\tilde{\theta}_2,
 \end{equation}
 for certain $\tilde{\theta}_1, \tilde{\theta}_2 \in L^2(\Omega)$.
 
 This amounts to saying that both $\alpha {_aD_x^{-(1-\mu)}}u$ and $\beta{_xD_b^{-(1-\mu)}u}$ are absolutely continuous on $\overline{\Omega}$, and so is $\alpha {_aD_x^{-(1-\mu)}}u+\beta{_xD_b^{-(1-\mu)}u}$.
 
So,  now taking the integration by ${_aD_x^{-1}}$ at both sides of \eqref{equ:taking-integ} yields
 \begin{equation}\label{equ:ContinueToRaiseTheRegularity}
 \alpha {_aD_x^{-(1-\mu)}}u+\beta{_xD_b^{-(1-\mu)}u}\overset{a.e.}{=}{_aD_x^{-1}}(L_2(u)+\frac{p}{k}u)+C,
 \end{equation}
  $C$ is another certain constant.
  
 Before going any further, it is worthwhile to observe that the right-hand side of \eqref{equ:ContinueToRaiseTheRegularity} is ``better" than the left-hand side in the sense that ${_aD_x^{-1}}(L_2(u)+\frac{p}{k}u)$ has at least integration of order $1$, plus a constant $C$, meanwhile, the left-hand side has only fractional integrations of order $1-\mu$. Since this is an identity, the two sides should behave the ``same", which seemingly suggests that $u$ in the left-hand side is supposed to possess at least fractional integration of order $\mu$ to balance the right-hand side. However, note that both sides are essentially  in terms of $u$ except the constant $C$, if $u$ is of fractional integration of order $\mu$ on the left, then  ${_aD_x^{-1}}(L_2(u)+\frac{p}{k}u)$ on the right is of fractional integration of order $1+\mu$, plus $C$, which is always ``better" than the left-hand side by order $\mu$. This way, we can keep continuing to raise the regularity of $u(x)$ until some threshold (if any) is reached.
 
  Let us come back to  \eqref{equ:ContinueToRaiseTheRegularity} and put this idea into action in the next two steps. 
  
  \vspace{0.2cm}
  
  4. We assert that $u(x)$ in \eqref{equ:ContinueToRaiseTheRegularity} can be represented as
 \begin{equation}
u(x)={_aD_x^{-\mu}}J, \quad \text{where} \quad J(x)\in H^*(\Omega).
\end{equation}

 To see this, we start with showing that, in the right-hand side of \eqref{equ:ContinueToRaiseTheRegularity}, $L_2(u)+\frac{p}{k}u \in H(\overline{\Omega})$.
 
 Check $\frac{p}{k}u$ first,  in view of \eqref{expression of u} and Property \ref{pro:MappingIntoHolderianFromLp}, 
 \[
 u \in H^{\mu/2}(\overline{\Omega}).
 \]
 Hence,
 \[
 \frac{p}{k}u \in H^{\mu/2}(\overline{\Omega}) \quad \text{since}\quad \frac{p}{k}\in C^1(\overline{\Omega}). 
 \]
 
 Similarly, for  each term of $L_2(u)$ by recalling the expression \eqref{equ:L2Expression}, it can be directly checked that 
 \[
 \begin{aligned}
 &\alpha{_xD_b^{-1}(\frac{k'}{k}{_aD_x^{\mu}}u)}
 , \beta {_xD_b^{-1}(\frac{k'}{k}{_xD_b^\mu}u)}, 
{_xD_b^{-1}}((\frac{p}{k})'u), {_xD_b^{-1}(\frac{q}{k}u)}\in H^{1/2}(\overline{\Omega}),\\
&\text{(we can go beyond $1/2$ in the exponent, but it is enough already.)}\\
&{_xD_b^{-1}}\frac{f}{k}\in H(\overline{\Omega}),  \quad
 C_1\in H^1(\overline{\Omega}),
 \end{aligned}
 \]
 by using \eqref{expression of u}, Property \ref{pro:MappingIntoHolderianFromLp},  conditions \eqref{conditionforthemainresult} and the fact that the quotient of two H$\ddot{\text{o}}$lderian functions is still H$\ddot{\text{o}}$lderian provided the denominator does not vanish. 
 
 Thus
 \begin{equation}\label{equ:UsingTheFirstLemmaToRaise}
 L_2(u)+\frac{p}{k}u \in H(\overline{\Omega}).
 \end{equation}

 The establishment of \eqref{equ:UsingTheFirstLemmaToRaise} permits us to apply Lemma \ref{lem-raising ragularity-1} to \eqref{equ:ContinueToRaiseTheRegularity}, which concludes our assertion that $u$ admits a representation of the form
 \begin{equation}\label{equ:leftrepresentation}
 u(x)={_aD_x^{-\mu}}J, \quad \text{where} \quad J(x)\in H^*(\Omega).
 \end{equation}
 According to  Property \ref{pro-correspondence-1}, $u(x)$ can also be represented by
  \[
 u(x)={_xD_b^{-\mu}}\tilde{J}, \quad \text{where} \quad \tilde{J}(x)\in H^*(\Omega).
 \]
It follows that
    \begin{equation}\label{equ:CrucialStep}
 {_aD_x^{\mu}}u, {_xD_b^{\mu}}u\in H^*(\Omega).
 \end{equation}


\vspace{0.2cm}

 5.  We claim that, for any given $0<\sigma<1$, the solution $u$ in \eqref{equ:ContinueToRaiseTheRegularity} can be represented as 
 \begin{equation}\label{equ:claimequation}
 u(x)={_aD_x^{-(\mu +\sigma)}}K_\sigma,
 \end{equation}
 where function $K_\sigma(x)$ belongs to $H^*(\Omega)$ and depends on $\sigma$.
 
 Substituting \eqref{equ:CrucialStep}, \eqref{equ:leftrepresentation} into $L_2(u)$ and \eqref{equ:leftrepresentation} into $\frac{p}{k}u$, we  deduce that 
\begin{equation}\label{equ:Firstcrucialapplication}
 L_2(u)+\frac{p}{k}u  \in H^*_{\mu}(\Omega).
\end{equation}
Indeed, it is directly checked that  
 \begin{equation}\label{equ:alreadyokay}
  L_2(u)\in  H^*_{t}(\Omega), \forall \ 0<t<1,\, \text{i.e.,}\, L_2(u)\in \bigcap_{0<t<1}H^*_t(\Omega)
 \end{equation}

 and that
 \[
 \frac{p}{k}u  \in  H^*_{\mu}(\Omega),
 \]
 by using conditions \eqref{conditionforthemainresult},  \eqref{equ:leftrepresentation}, \eqref{equ:CrucialStep}, Property \ref{pro-correspondence-1} and the fact that the product of two H$\ddot{\text{o}}$lderian functions is still H$\ddot{\text{o}}$lderian.
 
 \eqref{equ:Firstcrucialapplication} gives us the permission to apply Lemma \ref{lem: RaisingTheRegularity} to equation \eqref{equ:ContinueToRaiseTheRegularity}, from which we raise the fractional integration order of $u$ from $\mu $ to $\mu+\mu$, namely, 
 \begin{equation}\label{equ:Ragularitymu}
  u(x)={_aD_x^{-(\mu+\mu)}}K_\mu={_aD_x^{-2\mu}}K_\mu, \, \text{for a certain}\,  K_\mu \in H^*(\Omega).
 \end{equation}

At this stage, if $\mu\geq \sigma$, the claim \eqref{equ:claimequation}  has already been achieved due to the knowledge that $H^*_{y_1}(\Omega)\subseteq H^*_{y_2}(\Omega)$ for $0<y_2\leq y_1<1$.

 Otherwise, by a second substitution  of \eqref{equ:Ragularitymu} into $\frac{p}{k}u $ and noting \eqref{equ:alreadyokay}, we obtain
 \begin{equation}
 L_2(u)+\frac{p}{k}u  \in H^*_{\nu}(\Omega), \quad \nu=\min\{2\mu, \sigma\}.
 \end{equation}
 And then by a second application of Lemma \ref{lem: RaisingTheRegularity} to \eqref{equ:ContinueToRaiseTheRegularity}, we have
 \begin{equation}\label{equ:Ragularity2mu}
 u(x)={_aD_x^{-(\mu+\nu)}}K_{\nu}, \, \text{for a certain}\, K_{\nu} \in H^*(\Omega).
 \end{equation}

 Again, if $2\mu \geq \sigma$, the claim \eqref{equ:claimequation} has been confirmed. Otherwise, repeating, from now on, this procedure $n$ times, where $n$ is the smallest integer satisfying $(2+n)\mu \geq \sigma$, we obtain
  \begin{equation}
 L_2(u)+\frac{p}{k}u  \in H^*_{\sigma}(\Omega), \text{and}\,    u(x)={_aD_x^{-(\mu +\sigma)}}K_\sigma, K_\sigma\in H^*(\Omega).
 \end{equation}

(Notice: it is not concluded that $L_2(u)+\frac{p}{k}u  \in H^*_{(2+n)\mu}(\Omega)$ and $ u(x)={_aD_x^{-(\mu +(2+n)\mu)}}K_{(2+n)\mu}$ since $(2+n)\mu$ can be equal or beyond $1$, which can not guarantee the applicability of  Lemma \ref{lem: RaisingTheRegularity}, this is the essential difficulty that stops us reaching the extreme value $1+\mu$ in Theorem \ref{theorem}.)

 Thus, we have concluded  our claim \eqref{equ:claimequation}, which amounts to saying that for any given $0<t<1+\mu$, $u(x)$ can be represented as
 \[
 u(x)={_aD_x^{-t}J_t},\quad \text{for a certain}\quad 	J_t(x) \in H^*(\Omega).
 \]

\vspace{0.2cm}

 6. For  $t\leq 0$,  ${_aD_x^{-t}}$ denotes the fractional derivatives  or identity operators  (see notation in Section \ref{Notations}) and 
 	\[
 	u(x)={_aD_x^{-t}}({_aD_x^{t}}u),\, {_aD_x^{t}}u\in H^*(\Omega), x\in \Omega
 	\]
 	is always true.
 	
 	Therefore, for any  $t<1+\mu$, $u(x)$ is representable by
 	\begin{equation}\label{equ:weaksolutionu}
 	 	u(x)={_aD_x^{-t}J_t},\quad \text{for a certain}\quad 	J_t(x) \in H^*(\Omega).
 	\end{equation}

 	\vspace{0.2cm}
 	
 7. It remains to prove that $u(x)$ in  \eqref{equ:weaksolutionu} is a true solution to problem \eqref{equationformainresult} and is unique. 
 	
 To see $u(x)$ is a true solution, we need to show that $u\in C(\overline{\Omega})$, $u(a)=u(b)=0$, $Du\in C(\Omega)$, $(\alpha\, {_aD_x}^{-(1-\mu)}+\beta\,{_xD_b}^{-(1-\mu)})Du \in C^1(\Omega)$ and $[L(u)](x)= f(x)$  $\forall x\in \Omega$.	
 	
First,  from \eqref{equ:weaksolutionu},  $u(x)\in AC(\overline{\Omega})$ and $Du\in C(\Omega)$ are satisfied. Second, since $u(x)\in \widehat{H}^{(1+\mu)/2}_0(\Omega)$, $u(a)=u(b)=0$ is ensured.

This allows us to interchange the order of differentiation and fractional integrations to obtain
\[
(\alpha\, {_aD_x}^{-(1-\mu)}+\beta\,{_xD_b}^{-(1-\mu)})Du=D(\alpha\, {_aD_x}^{-(1-\mu)}+\beta\,{_xD_b}^{-(1-\mu)})u,
\]
by invoking \eqref{equ:taking-integ}  we hence see
\begin{equation}\label{equ:Forinvokingincorollary}
\alpha {_aD_x^{-(1-\mu)}}Du+\beta{_xD_b^{-(1-\mu)}}Du=L_2(u)+\frac{p}{k}u  \in C^1(\Omega).
\end{equation}
 
 Differentiating both sides of \eqref{equ:Forinvokingincorollary}, substituting for $L_2(u)$, then multiplying both sides by $k(x)$ and simplifying, we recover FDARE pointwisely, namely
 \[
 [L(u)](x)=f(x), \text{for every point}\, x\in \Omega.
 \]	
 	
 Hence, $u(x)$ in \eqref{equ:weaksolutionu} is a true solution (classical solution) to problem \eqref{equationformainresult}.
 	
 \vspace{0.2cm}
 	
 8.  Lastly, we  prove the uniqueness. Suppose that there is another function $v(x)\in \widehat{H}^{s}_0(\Omega)$ such that it has a representative which is also a true solution to problem \eqref{equationformainresult}, we show $v$ has to coincide with $u$.
 
Multiplying both sides of 
 \[
 \frac{L(v)}{k}=\frac{f}{k}
 \]
  by an arbitrary $\psi(x)\in C^\infty_0(\Omega)$ and integrating over $\Omega$, we have
  \begin{equation}\label{equ:lefthandneeded}
    \int_\Omega \frac{[L(v)](t)}{k(t)}\, \psi(t)\, dt=\int_\Omega \frac{f(t)}{k(t)} \, \psi(t)\, dt.
  \end{equation}

  Since $v(x)$ is a true solution, by definition, $v\in AC(\overline{\Omega})$, $v(a)=v(b)=0$, which implies
  	\begin{equation}\label{equ:thesameexpression}
  	  D(\alpha\, {_aD_x^{-(1-\mu)}}+\beta\,{_xD_b^{-(1-\mu)}})v=(\alpha\, {_aD_x^{-(1-\mu)}}+\beta\,{_xD_b^{-(1-\mu)}})Dv.
  	\end{equation}

  Simplifying and manipulating the left-hand side of \eqref{equ:lefthandneeded} by taking advantage of  \eqref{equ:thesameexpression} and the assumption $v(x)\in \widehat{H}^{s}_0(\Omega)$ yields 
  \[
  B_2[v,\psi]=(\frac{f}{k}, \psi)_\Omega, \psi\in C_0^\infty(\Omega).
  \]
  
 Therefore, $v(x)$ is also a weak solution to  \eqref{equ:variationalformulation}, which implies $u(x)=v(x)$  a.e. by  the uniqueness of weak solution in $\widehat{H}^s_0(\Omega)$.
 
 The whole proof of Theorem \ref{theorem} is completed.
\end{proof}
\subsection{An example}
We provide an example to show that the value $1+\mu$ in Theorem  \ref{theorem} is optimal, namely, for any given $ t>1+\mu$, there exists a true solution $u(x)$ to \eqref{equationformainresult} under conditions \eqref{conditionforthemainresult} such that $u(x)$ is not representable by  ${_aD_x^{-t}J_t}$, namely,  $	u(x)\neq {_aD_x^{-t}J_t}$ for any $J_t(x) \in H^*(\Omega)$.

Let us consider 
	\begin{equation}
\begin{cases}
[L(u)](x)= -1,\, x\in \Omega,\\
u(a)=u(b)=0,\\
[L(u)](x) := -D(\alpha\, {_aD_x^{-(1-\mu)}}+\beta\,{_xD_b^{-(1-\mu)}})Du,\\
\end{cases}
\end{equation}
it is clear that conditions in \eqref{conditionforthemainresult} are satisfied.
From Lemma \ref{lem-solution-2}, we know the true solution is
$u(x)=c\cdot (x-a)^{p+1}(b-x)^{q+1}$, where 
\begin{equation}
c=\frac{(-p-1+\mu)(-p+\mu)\Gamma(1-\mu)}{\mu(1+\mu)\beta \,\Gamma(2-\mu+p)\Gamma(q+2)},
\end{equation}
and  $p,q$ are uniquely determined by 
\begin{equation}\label{equ:conditionsforexample}
p+q=-(1-\mu) \quad \text{and}\quad
\alpha \sin(q\pi)=\beta\sin(p\pi).
\end{equation}

Using conditions \eqref{equ:conditionsforexample}, we carry out
\begin{equation}\label{equ:example1}
\frac{\beta}{\alpha}=\cot(-p\pi)\sin((1-\mu)\pi)-\cos((1-\mu)\pi).
\end{equation}
In fact, if $\frac{\beta}{\alpha}$ (similarly for $\frac{\alpha}{\beta}$) is small enough, for example, let $\frac{\beta}{\alpha}$ satisfies 
\begin{equation}\label{equ:example2}
0<\frac{\beta}{\alpha}<\cot((2-t_0)\pi)\sin((1-\mu)\pi)-\cos((1-\mu)\pi),
\end{equation}
where $1+\mu<t_0<\min\{2, t\}$,  then  \eqref{equ:conditionsforexample}, \eqref{equ:example1} and \eqref{equ:example2} produce
\begin{equation}\label{examplecondition}
p+1-t<-1.
\end{equation}

It is directly evaluated by using Property \ref{coupled-inte} that 
\[
{_aD_x^t}u= C\cdot (x-a)^{p+1-t}\, {_2F_1}(-1-q,p+2,p+2-t,\frac{x-a}{b-a}),\, a<x<b,
\]
where $C$ is a non-zero constant.

We see that ${_aD_x^t}u \notin L^1(\Omega)$ due to \eqref{examplecondition}, which implies that $u(x)$  can not be represented as ${_aD_x^{-t}J_t}$ for any $J_t(x) \in H^*(\Omega)$. 
\section{Applications}\label{sec:applications}
We conclude this paper by giving two  corollaries from a practical point of view.

\subsection{FDARE with Riemann-Liouville derivative}

The FDARE with Riemann-Liouville derivative has also attracted the attention of many authors in modelling, namely,
	\begin{equation}\label{equationformainresult-RL}
\begin{cases}
[\tilde{L}(u)](x)= f(x),\, x\in \Omega,\\
u(a)=u(b)=0,\\
[\tilde{L}(u)](x) := -Dk(x)D(\alpha\, {_aD_x^{-(1-\mu)}}+\beta\,{_xD_b^{-(1-\mu)}})u\\
\qquad \qquad+p(x)Du+q(x)u(x).
\end{cases}
\end{equation}
We say a function $u(x)$ is a true (classical) solution to \eqref{equationformainresult-RL} if $u\in AC(\overline{\Omega})$, $u(a)=u(b)=0$, $Du\in C(\Omega)$, $D(\alpha\, {_aD_x}^{-(1-\mu)}+\beta\,{_xD_b}^{-(1-\mu)})u \in C^1(\Omega)$ and $[\tilde{L}(u)](x)= f(x)$  $\forall x\in \Omega$.

The following corollary clarifies that the FDARE with either Riemann-Liouville derivative or Caputo derivative always has the same  solution.
\begin{corollary}\label{cor:1}
	Let conditions \eqref{conditionforthemainresult} be satisfied. Then in the Sobolev space $\widehat{H}^{(1+\mu)/2}_0(\Omega)$, problem \eqref{equationformainresult} and  \eqref{equationformainresult-RL} have the same true solution.
\end{corollary}
\begin{proof}
First,	according to the definition of true  solution,  we readily verify that $u(x)$ is a true solution of \eqref{equationformainresult} if and only if $u(x)$ is a true solution of  \eqref{equationformainresult-RL}  by noting that
	\[
	D(\alpha\, {_aD_x^{-(1-\mu)}}+\beta\,{_xD_b^{-(1-\mu)}})u=(\alpha\, {_aD_x^{-(1-\mu)}}+\beta\,{_xD_b^{-(1-\mu)}})Du
	\]
if  $u\in AC(\overline{\Omega})$, $u(a)=u(b)=0$.

Second,  the existence of such a true solution directly follows from Theorem \ref{theorem}.

\end{proof}
\subsection{More structure information on the solution}
In practice, one is also  interested in knowing what happens at the boundary points of the $\mu$th-order derivative and the first-order derivative of the solution, namely, ${_aD_x^\mu}u$,  ${_xD_b^\mu}u$ and $Du$, which is essentially related to the well-posedness problems of FDARE associated with other types of conditions. The following corollary gives  refined structures of the solution by imposing a little bit stronger conditions, namely,
	\begin{equation}\label{cond:1}
\begin{cases}
0<\alpha, \beta <1, \alpha+ \beta=1, 1/2<\mu<1,\\
f\in H^*_{1-\mu}(\Omega), p(x),q(x), k(x)\in C^2(\overline{\Omega}), \\
k(x)>0 \, \text{on}\, \overline{\Omega}, \frac{q}{k}-\frac{1}{2}(\frac{p}{k})'\geq 0 \, \text{on}\, \overline{\Omega},\\
\pi(1-\mu^2)\cot((1+\mu)\pi/2) +4(b-a)\|\frac{k'}{k}\|_{L^\infty(\Omega)}<0,
\end{cases}
\end{equation}

or 
	\begin{equation}\label{cond:2}
\begin{cases}
0<\alpha, \beta <1, \alpha+ \beta=1, 0<\mu<1,\\
f\in H^*_{1-\mu}(\Omega), q(x), k(x)\in C^2(\overline{\Omega}), \\
p(x)=0, x\in \overline{\Omega},\\
k(x)>0 \, \text{on}\, \overline{\Omega}, \frac{q}{k}-\frac{1}{2}(\frac{p}{k})'\geq 0 \, \text{on}\, \overline{\Omega},\\
\pi(1-\mu^2)\cot((1+\mu)\pi/2) +4(b-a)\|\frac{k'}{k}\|_{L^\infty(\Omega)}<0.
\end{cases}
\end{equation}

\begin{corollary}\label{cor:Thesecondone}
	Let conditions \eqref{cond:1} or  \eqref{cond:2} (or both!) be satisfied.  Then  in $\widehat{H}^{(1+\mu)/2}_0(\Omega)$, there exists a unique true solution $u(x)$ (up to the equivalence classes) to problem \eqref{equationformainresult} and it satisfies the following:
	\begin{enumerate}
		\item  ${_aD_x^\mu}u, {_xD_b^\mu}u \in C(\overline{\Omega})$ and	
		\begin{equation}\label{equ:FDZ}
		   {_aD_x^\mu}u \arrowvert_{x=a}={_xD_b^\mu}u\arrowvert_{x=b}=0.
		\end{equation}
		\item  $Du\in C(\Omega)$  and precisely one of the following holds:
	\begin{enumerate}
		\item 		\begin{equation}\label{equ:DZ}
\lim_{x\rightarrow a^+}Du=0,\quad \lim_{x\rightarrow b^-}Du=0,
	\end{equation}
	\item 	\begin{equation}\label{equ:DZz}
\lim_{x\rightarrow a^+}Du=0,\quad \lim_{x\rightarrow b^-}|Du|=+\infty,
		\end{equation}
\item 	\begin{equation}\label{equ:DZzz}
\lim_{x\rightarrow a^+}|Du|=+\infty,\quad \lim_{x\rightarrow b^-}Du=0,
		\end{equation}
\item 	\begin{equation}\label{equ:DZzzz}
\lim_{x\rightarrow a^+}|Du|=+\infty,\quad \lim_{x\rightarrow b^-}|Du|=+\infty.
		\end{equation}
			\end{enumerate}
	\end{enumerate}
\end{corollary}
\begin{proof}
	Let us see that, compared to Theorem \ref{theorem}, stronger conditions are imposed in Corollary \ref{cor:Thesecondone}, namely, either  \eqref{cond:1} or \eqref{cond:2} ensures \eqref{conditionforthemainresult}; therefore, all the justifications in the proof of Theorem \ref{theorem} remain valid to Corollary \ref{cor:Thesecondone}.
	
1.  The existence and uniqueness of the true solution $u(x)$ is an immediate consequence of Theorem \ref{theorem}. And it is representable by 
\begin{equation}\label{equ:Beinvokedforcorollary2}
u(x)= {_aD_x^{-t}J_t},
\end{equation}
$J_t(x) \in H^*(\Omega)$ (depending on $t$)  provided that $t<1+\mu$.

Taking into account $u(a)=u(b)=0$ and  Property \ref{pro-correspondence-1}, we also have
\begin{equation}\label{equ:Beinvokedforcorollary3}
u(x)={_xD_b^{-t}\tilde{J}_t},
\end{equation}
$\tilde{J}_t(x) \in H^*(\Omega)$ (depending on $t$)  provided that $t<1+\mu$.

It remains to show part $(1)$ and part $(2)$.

2. Let us invoke the identity in \eqref{equ:Forinvokingincorollary}, namely,
\begin{equation}\label{equ:beforederivative}
\alpha {_aD_x^{-(1-\mu)}}Du+\beta{_xD_b^{-(1-\mu)}}Du=L_2(u)+\frac{p}{k}u.
\end{equation}
	
	Recalling the expression of $L_2(u)$ and differentiating the right-hand side give
			\begin{equation}\label{equ:expressionforL2}
	D(L_2(u)+\frac{p}{k}u)=-\alpha\frac{k'}{k}{_aD_x^{\mu}}u+\beta \frac{k'}{k}{_xD_b^\mu}u	+\frac{q}{k}u-\frac{f}{k}+\frac{p}{k}Du.
	\end{equation}

	From \eqref{equ:Beinvokedforcorollary2} and \eqref{equ:Beinvokedforcorollary3}, we see
	\begin{equation}\label{equ:muthderivative}
		{_aD_x^{\mu}}u,	{_xD_b^{\mu}}u\in H^*_{\sigma}(\Omega) \quad \forall  \,0<\sigma<1.
	\end{equation}
	
	Substituting \eqref{equ:muthderivative} into \eqref{equ:expressionforL2} and checking each term by utilizing either conditions \eqref{cond:1} or conditions \eqref{cond:2}, it is guaranteed that
	 \[
	D(L_2(u)+\frac{p}{k}u)\in H^*_{1-\mu}(\Omega). 
	\]

In view of Property \ref{Existence-Abel}, we know there exists a function $v(x)\in H^*(\Omega)$ such that
\begin{equation}\label{equ:afterderivative}
\alpha {_aD_x^{-(1-\mu)}}v+\beta{_xD_b^{-(1-\mu)}}v=D(L_2(u)+\frac{p}{k}u).
\end{equation}

Once we have \eqref{equ:beforederivative} and \eqref{equ:afterderivative}, according to Lemma \ref{lem:solution-derivative}, we know 
\begin{equation}\label{derivative-remove}
D(\alpha{_aD_x^{-(1-\mu)}}+\beta{_xD_b^{-(1-\mu)}})(Du-Y)=0, x\in \Omega,
\end{equation}
where 
\begin{equation}\label{equ:coefficients}
\begin{aligned}
Y&=\int_a^x v(t)\, dt-\frac{cS}{S_1}\int_a^x (t-a)^p(b-t)^q\, dt+\frac{c_1S}{S_1} D((x-a)^{p+1}(b-x)^{q+1}),\\
S&= \int_a^bv(t) \, dt,\\
c&=  \Gamma(-q)\left(\alpha(b-a)^{1+p+q}\Gamma(p+1)\int_a^b (t-a)^{-q-1}(b-t)^{-p-1}\, dt \right)^{-1},\\
S_1&=  \frac{\Gamma(-q)\int_a^b (t-a)^p(b-t)^q \, dt}{ \alpha(b-a)^{1+p+q}\Gamma(p+1)\int_a^b (t-a)^{-q-1}(b-t)^{-p-1}\, dt },\\
c_1&= \frac{(-p-1+\mu)(-p+\mu)\Gamma(1-\mu)}{\mu(1+\mu)\beta \,\Gamma(2-\mu+p)\Gamma(q+2)},
\end{aligned}
\end{equation}
and $p$, $q$ are uniquely determined by
\begin{equation}
p+q=-(1-\mu) \quad \text{and}\quad
\alpha \sin(q\pi)=\beta \sin(p\pi).
\end{equation}

3. We would like to get rid of the derivative operator $D$ at the left-hand side of ~\eqref{derivative-remove}. Observe that
\[
(\alpha{_aD_x^{-(1-\mu)}}+\beta{_xD_b^{-(1-\mu)}})(Du-Y)\in C^1(\Omega),
\]
 as a result, 
	\begin{equation}\label{equ:Last1}
	(\alpha{_aD_x^{-(1-\mu)}}+\beta{_xD_b^{-(1-\mu)}})(Du-Y)=c_2,
	\end{equation}
	for a certain constant $c_2$.
	
 On the other hand, taking  Lemma~\ref{lem-solution-1} into account, we know 
	\begin{equation}\label{equ:Last2}
	(\alpha{_aD_x^{-(1-\mu)}}+\beta{_xD_b^{-(1-\mu)}})(c_3\cdot(x-a)^p(b-x)^q)=c_2,
	\end{equation}
	for a suitable constant $c_3$.
	
	\eqref {equ:Last1} subtracting \eqref{equ:Last2} yields
	\begin{equation}
	(\alpha{_aD_x^{-(1-\mu)}}+\beta{_xD_b^{-(1-\mu)}})(Du-Y_1)=0,
	\end{equation}
	where 
	\begin{equation}
	Y_1=Y+c_3\cdot(x-a)^p(b-x)^q.
	\end{equation}
	Since $Du-Y_1\in H^*(\Omega)$, we conclude by Property~\ref{Existence-Abel} that
	\begin{equation}
	Du-Y_1=0, 
	\end{equation}
	namely, 
	\begin{equation}\label{equ:Lastcombination}
	Du(x)=\Pi_1(x)+\Pi_2(x), \, x\in \Omega,
	\end{equation}
	where
	\begin{equation}\label{equ:expression}
	\begin{aligned}
	\Pi_1(x)&=\int_a^x v(t)\, dt-\frac{cS}{S_1}\int_a^x (t-a)^p(b-t)^q\, dt,\\
	\Pi_2(x)&=\frac{c_1S}{S_1} D(x-a)^{p+1}(b-x)^{q+1}+c_3\cdot(x-a)^p(b-x)^q.
	\end{aligned}
	\end{equation}
	
4. Let us prove part $(1)$ now.

 First we claim $Du\in H^*_\mu(\Omega)$.

To see this, we just need to show
\[
\Pi_1(x), \Pi_2(x) \in H^*_\mu(\Omega).
\]

$\Pi_1(x)\in H^*_\mu(\Omega)$ is clear. To see $\Pi_2(x)\in H^*_\mu(\Omega)$, expanding
\begin{equation}\label{equ:threetherms}
\begin{aligned}
\Pi_2(x)&=\frac{c_1(p+1)S}{S_1} (x-a)^{p}(b-x)^{q+1}-\frac{c_1(q+1)S}{S_1} (x-a)^{p+1}(b-x)^{q}\\
&+c_3\cdot(x-a)^p(b-x)^q.
\end{aligned}
\end{equation}

 For the first term, manipulating $(x-a)^{p}(b-x)^{q+1}$ as
\[
\begin{aligned}
(x-a)^{p}(b-x)^{q+1}&=\frac{(x-a)^{p+1+q+\epsilon}(b-x)^{q+1}}{(x-a)^{1+q+\epsilon}}\\
&=\frac{(x-a)^{\mu+\epsilon}(b-x)^{-p+\mu}}{(x-a)^{1-(-q-\epsilon)}}\\
&\text{(using the condition $p+q=-(1-\mu)$)},
\end{aligned}
\]
where $\epsilon$ is chosen so that $0<\epsilon<\min\{-q, -p\}$. In the numerator, we see 
\[
(x-a)^{\mu+\epsilon}(b-x)^{q+1}\in H^{\mu+\epsilon}_0(\overline{\Omega}) 
\]
with the auxiliary inequality
 \begin{equation}
\frac{|y_1^y-y_2^y|}{|y_1-y_2|^y}\leq 1,\, (0\leq y\leq 1, 0<y_1, 0<y_2, y_1\neq y_2).
\end{equation}
Therefore, by the definition of $H^*_\mu(\Omega)$, $(x-a)^{p}(b-x)^{q+1}\in H^*_\mu(\Omega)$.

As regards the second and third term in \eqref{equ:threetherms}, it can be analogously verified, respectively. This concludes our claim $Du\in H^*_\mu(\Omega)$, which in turn implies,  in view of Property \ref{pro-correspondence-1}, that
\[
Du={_aD_x^{-\mu}}J_1={_xD_b^{-\mu}}(-J_2),\, \text{for certain}\, J_1(x), J_2(x)\in H^*(\Omega).
\]

Taking $u(a)=u(b)=0$ into account, we obtain
\[
u(x)={_aD_x^{-1}}Du=-{_xD_b}^{-1}Du={_aD_x^{-(1+\mu)}}J_1={_xD_b^{-(1+\mu)}}J_2.
\] 

So,
\[
 {_aD_x^\mu}u=\int_a^x J_1(t)\, dt, \quad {_xD_b^\mu}u=\int_x^b J_2(t)\, dt,
\]
from which we see ${_aD_x^\mu}u, {_xD_b^\mu}u \in C(\overline{\Omega})$ and	
\[
{_aD_x^\mu}u \arrowvert_{x=a}={_xD_b^\mu}u\arrowvert_{x=b}=0.
\]

5. Last, we prove part $(2)$.

Let us take a closer look at $\Pi_1$ and $\Pi_2$ in \eqref{equ:expression}.

Clearly, $Du \in C(\Omega)$ and $\Pi_1(a)=0$. A direct calculation by substituting for $c,S, S_1$ into $\Pi_1(x)$ from \eqref{equ:coefficients} yields $\Pi_1(b)=0$. 

If we can show that $\Pi_2(x)$ either vanishes or blows up at the boundary point $x=a$ and  either vanishes or blows up at the boundary point $x=b$, then part $(2)$ is proved.

Let us first show that $\Pi_2(x)$ either vanishes or blows up at $x=a$.

To see this, according to the  expression of $\Pi_2(x)$ in \eqref{equ:expression}, if $ \frac{c_1S}{S_1}=0$, it is clear that $\Pi_2(x)$ either vanishes or blows up at $x=a$ by noting $p,q<0$, depending on whether $c_3$ disappears.

If $\frac{c_1S}{S_1}\neq0$, using the other expression in \eqref{equ:threetherms} and  operating with the condition $p+q=-(1-\mu)$, we have
\[
\begin{aligned}
\Pi_2(x)&=\frac{c_1(p+1)S (b-x)^q}{S_1}\left(b-a+\frac{c_3 S_1}{c_1(p+1)S}\right)(x-a)^p\\
&-\frac{c_1(1+\mu) S}{S_1} (x-a)^{p+1}(b-x)^{q},
\end{aligned}
\]

and again we see
\[
\text{either}\quad \lim_{x\rightarrow a^+}\Pi_2(x)=0\quad \text{or}\quad \lim_{x\rightarrow a^+}|\Pi_2(x)|=+\infty,
\]
by noting $p<0$ and $p+1>0$, depending on whether $b-a+\frac{c_3 S_1}{c_1(p+1)S}$ disappears.

In a similar fashion, it can be shown the same at $x=b$. Namely,

 \[
 \text{either}\quad  \lim_{x\rightarrow b^-}\Pi_2(x)=0\quad \text{or}\quad \lim_{x\rightarrow b^-}|\Pi_2(x)|=+\infty.
 \]

Hence, part $(2)$ follows.

 The whole proof is completed.
\end{proof}

In Corollary \ref{cor:Thesecondone}, two key conditions have been imposed, i.e., $1/2<\mu<1$ and $p(x)=0,\, x\in \Omega$, which corresponds to the diffusion order and the advection term, respectively. We want to know if they are intrinsically related to physical phenomena or not, or equivalently, from a mathematical point of view,  whether these constrains can be removed. It turns out that it  is connected to the extreme value $1+\mu$ in Theorem \ref{theorem} in a certain sense and therefore we propose the following question:

Q: Under conditions \eqref{conditionforthemainresult}, can each true solution of problem \eqref{equationformainresult} be represented as $u(x)={_aD_x^{-(1+\mu)}}J, \, J(x)\in H^*(\Omega)$?
%
%
%
%
%


%


\bibliographystyle{siam}
\bibliography{ftpbvp}

\begin{thebibliography}{1}

\bibitem{MR1688958}
{\sc G.~E. Andrews, R.~Askey, and R.~Roy}, {\em Special functions}, vol.~71 of
  Encyclopedia of Mathematics and its Applications, Cambridge University Press,
  Cambridge, 1999.

\bibitem{MR3802435}
{\sc V.~J. Ervin, N.~Heuer, and J.~P. Roop}, {\em Regularity of the solution to
  1-{D} fractional order diffusion equations}, Math. Comp., 87 (2018),
  pp.~2273--2294.

\bibitem{ernmpde06}
{\sc V.~J. Ervin and J.~P. Roop}, {\em Variational formulation for the
  stationary fractional advection dispersion equation}, Numer. Methods Partial
  Differ. Equ., 22 (2006), pp.~558--576.

\bibitem{glsima18}
{\sc V.~Ginting and Y.~Li}, {\em On the fractional diffusion-advection-reaction
  equation in {$\Bbb{R}$}}, Fract. Calc. Appl. Anal., 22 (2019),
  pp.~1039--1062.

\bibitem{MR4024334}
{\sc Y.~Li}, {\em On {F}ractional {D}ifferential {E}quations and {R}elated
  {Q}uestions}, ProQuest LLC, Ann Arbor, MI, 2019.
\newblock Thesis (Ph.D.)--University of Wyoming.

\bibitem{MR1347689}
{\sc S.~G. Samko, A.~A. Kilbas, and O.~I. Marichev}, {\em Fractional integrals
  and derivatives}, Gordon and Breach Science Publishers, Yverdon, 1993.

\end{thebibliography}
 \bigskip \smallskip
 \it
 \noindent
Science and Mathematics Cluster\\
Singapore University of Technology and Design\\
8 Somapah Road Singapore 487372, Singapore\\
e-mail: liyulong0807101@gmail.com \\
    
\end{document}